\documentclass[a4paper,11pt]{amsart}
\usepackage[utf8]{inputenc}
\usepackage{amsmath, amssymb}
\usepackage{setspace}
\usepackage[left=3cm, right=3cm, top=3cm, bottom=3cm]{geometry}                 

\usepackage{graphicx} 
\usepackage{pgf,tikz}
\usetikzlibrary{arrows}
\usepackage{amssymb}
\usepackage{pdfsync}
\usepackage{mathrsfs}
\usepackage{hyperref} 
\usepackage{verbatim} 
\usepackage{epstopdf}
\usepackage{xargs}
\usepackage{todonotes}
\usepackage[normalem]{ulem}

\newcommandx{\jow}[2][1=]{\todo[linecolor=orange,backgroundcolor=orange!25,bordercolor=orange,#1]{#2}}

\newcommandx{\mateus}[2][1=]{\todo[linecolor=blue,backgroundcolor=blue!25,bordercolor=blue,#1]{#2}}

\def\R{\mathbb{R}}

\def\N{\mathbb{N}}
\def\Z{\mathbb{Z}}

\newcommand{\eps}{\varepsilon}
\newcommand{\mmd}{\mathrm{d}}
\newcommand{\Xstmax}{X^s_T(\text{max})}
\newcommand{\Xstsmooth}{X^s_T(\text{smooth})}
\newcommand{\Xstsobolev}{X^s_T(\text{Sobolev})}
\newcommand{\Xststrichartz}{X^s_T(\text{Strichartz})}

\newcommandx{\eq}{\approxeq}

\newtheorem{theorem}{Theorem}
\newtheorem{corollary}[theorem]{Corollary}

\newtheorem{remark}[theorem]{Remark}
\newtheorem{proposition}[theorem]{Proposition}
\newtheorem{lemma}[theorem]{Lemma}

\def\Xint#1{\mathchoice
{\XXint\displaystyle\textstyle{#1}}%
{\XXint\textstyle\scriptstyle{#1}}%
{\XXint\scriptstyle\scriptscriptstyle{#1}}%
{\XXint\scriptscriptstyle\scriptscriptstyle{#1}}%
\!\int}
\def\XXint#1#2#3{{\setbox0=\hbox{$#1{#2#3}{\int}$}
\vcenter{\hbox{$#2#3$}}\kern-.5\wd0}}

\def\dashint{\Xint-}

\numberwithin{equation}{section}
\numberwithin{theorem}{section}

\begin{document}
\title[$L^2$ critical gZK equation in 3D]{The Cauchy problem for the $L^2-$critical generalized Zakharov-Kuznetsov equation in dimension 3}
\author{Felipe Linares}
\address[F. Linares]{IMPA\\
Instituto Matem\'atica Pura e Aplicada\\
Estrada Dona Castorina 110\\
22460-320, Rio de Janeiro, RJ\\Brazil}
\email{linares@impa.br} 
\author{Jo\~ao P.G. Ramos}
\address[J.P. Ramos]{IMPA\\
Instituto Matem\'atica Pura e Aplicada\\
Estrada Dona Castorina 110\\
22460-320, Rio de Janeiro, RJ\\Brazil}
\email{joaopgramos95@gmail.com}

\keywords{Zakharov-Kuznetsov equation,  well-posedness,  low regularity}
\subjclass{Primary: 35Q53. Secondary: 35B05}

\begin{abstract}
We prove local well-posedness for the $L^2$ critical generalized Zakharov-Kuznetsov equation in $H^s, \, s \in (3/4,1).$ We also prove that the equation is ``almost well-posedness'' for
initial data $u_0 \in H^s, \, s \in [1,2),$ in the sense that the solution belongs to a certain intersection $C([0,T]\colon H^s(\R^3)) \cap X^s_T$ and is unique within that class, 
where we can ensure continuity of the data-to-solution map in an only slightly larger space. We also prove that solutions satisfy the expected conservation of $L^2-$mass for the whole $s \in (3/4,2)$ range, 
and energy for $s \in (1,2).$ By a limiting argument, this implies, in particular, global existence for small initial data in $H^1.$ Finally, we study the question of almost everywhere (a.e.) convergence of solutions of the initial value problem to initial data.
\end{abstract}
\maketitle

\section{Introduction}

In this paper we will study the initial value problem (IVP) associated to the  $L^2$ critical generalized Zakharov-Kuznetsov equation in
dimension 3, that is,
\begin{equation}\label{eq gZK 4/3}
\begin{cases}
\partial_t u + \partial_x \Delta u + \partial_x (u^{7/3}) = 0 & \text{ on } \R^3 \times \R; \cr 
u(x,0) = u_0(x) & \text{ on } \R^3. \cr
\end{cases}
\end{equation}
where $u$ is a real function and $\Delta$ denotes the Laplace operator in space variables.

The equation above is related to the family of the well-known generalized Zakharov-Kuznetsov equation
\begin{equation}\label{gen-zk}
\partial_tu+\partial_x\Delta u+u^k\partial_xu=0, \hskip10pt \text{ on }\; \R^d \times \R, \hskip5pt k\in \Z^{+},\\
\end{equation}
where $d=2, 3$.

The equation above, for $k=1$, arises in the  context of plasma physics, where it was formally derived by Zakharov and Kuznetsov
\cite{ZakharovKuznetsov}  as a long wave small-amplitude limit of the Euler-Poisson system in the ``cold plasma'' approximation. This formal long-wave 
limit was rigorously justified by  Lannes, Linares and Saut   in \cite{LannesLinaresSaut}  (see also \cite{H} for derivation in a different context). 

We also notice that the generalized Zakharov-Kuznetsov equation \eqref{gen-zk} is not completely integrable,
but it has a Hamiltonian structure and possesses three invariants, namely,
\begin{equation*}
	I(t)=\int_{\R^d} u(x,t)\,dx=I(0),
\end{equation*}
\begin{equation*}
	M(t)=\int_{\R^d}u^2(x,t)\,dx= M(0),
\end{equation*}
and
\begin{equation*}
E(t)=\int_{\R^d}\Big(|\nabla u|^2-\frac{2u^{k+1}}{(k+1)(k+2)}\Big)dx=E(0),
\end{equation*}
which are well known to be important in order to obtain a priori estimates that allows one to extend solutions of the IVP globally.

\vspace{3mm}

In addition, a scaling argument suggests that in order to obtain local well-posedness in $H^s(\R^d)$  for the IVP associated to equation \eqref{gen-zk} it is 
required to have $s \ge s_k= \frac{d}{2}-\frac{2}{k}$.

\vspace{3mm}

There is an extensive literature addressing the issue of local well-posedness of the IVP associated to the equation \eqref{gen-zk}.
For $k=1$ and 2D, it starts with the work of Faminskii \cite{Faminskii}, where the local theory for initial data in the Sobolev spaces $H^s(\R^2)$,
$s\ge1,$ is established. The method of proof used there combine smoothing estimates and a contraction mapping principle. We mention also some
subsequent extensions obtained by Linares and Pastor in \cite{LPa}, refining this method with inspiration in the results for the Korteweg-de Vries equation
\begin{equation}\label{kdv}
\partial_tu+\partial_x^3u+u\partial_xu=0, \hskip10pt \text{ on } \; \R \times \R
\end{equation}
by Kenig, Ponce and Vega \cite{KPV1}.  A major improvement was obtained by Molinet and Pilod \cite{MolinetPilod} and
Gr\"unrock and Herr \cite{GruenrockHerr}, in both cases using the Fourier restriction method, where the authors prove local well-posedness 
in $H^s(\R^2)$, $s>1/2$. We notice that the aforementioned results allow to extend the local solutions
globally in $H^s(\R^2), s\ge 1$. Recently, Kinoshita \cite{Kinoshita1} obtained the best possible result attainable with the method of contraction
principle. He proves a local theory in $H^s(\R^2)$, $s>-1/4$,  by implementing sharp bilinear estimates and using a rather
refined analysis. Kinoshita's result implies global well-posedness in $H^s(\R^2)$, $s\ge 0$. In the 2D case, the scale argument suggests local well-posedness results for \eqref{gen-zk} should hold for data in $H^s(\R^2)$, 
with $s>s_k=1-2/k$.  In that regard, sharp local results were obtained by Ribaud and Vento \cite{RV1} for $k\ge 4$.   
In \cite{Gru2}, Gr\"unrock proved the local well-posedness for \eqref{gen-zk} in the case $k=3$ in $H^s(\R^2)$, $s>1/3$.
For the particular nonlinearity $k=2$, also called modified Zakharov-Kuznetsov equation (mZK), for which  
$L^2(\R^2)$ is the critical space suggested by the scale, Kinoshita \cite{Kinoshita2} showed that the best possible result is
$\dot{H}^{s}(\R^2)$, $s\ge 1/4$, complementing the Result by Ribaud and Vento \cite{RV1} that proves local well-posedness in $H^s(\R^2)$, $s>1/4$. In addition to that, recently Bhattacharya, Farah and Roudenko  \cite{BFR} proved global well-posedness for the same 
modified Zakharov--Kuznetsov equation in 2D if $s >3/4$. For more related results see also \cite{BL}, \cite{LPa2}, and \cite{FLPa}.

Regarding the 3D problem and $k=1$, Ribaud and Vento \cite{RV2} showed local well-posedness in $H^s(\R^3)$, 
$s>1$ (see also \cite{LSaut}). Global well-posedness was proved by Molinet and Pilod \cite{MolinetPilod} in the same 
spaces. Herr and Kinoshita in \cite{HK} establish an optimal local well-posedness result by using Picard iteration method, 
for data in $H^s(\R^3)$, $s>-1/2$ and so obtaining global well-posedness in $H^s(\R^3)$, $s\ge 0$. For $k\ge 3$, results 
by Gr\"unrock  \cite{Gru2} establish local well-posedness  for data in $H^s(\R^3)$  and 
$s>3/2-2/k$. The case $k=2$ was particularly addressed by both Gr\"unrock \cite{Gru1} and Kinoshita \cite{Kinoshita2}, 
where both authors show that the IVP \eqref{gen-zk} is locally well-posed in $H^{s}(\R^3)$ for $s > 1/2.$ Kinoshita additionally shows that we can push the former result up 
to the endpoint, and global existence holds in the whole $s \ge 1/2$ range.

We notice that in the 3D case the scale of the equation in \eqref{gen-zk} does not attain $L^2$ for any power $k\in \Z^{+}$, differently
from the case in 2D. The natural question that arises is whether one can have a meaningful equation generalizing the ZK equation where we can
expect local well-posedness in $L^2$.  Using the scale and letting $k$ in \eqref{gen-zk} vary through the reals we find the equation \eqref{eq gZK 4/3}, 
where we consider real initial data $u_0,$ for which we investigate the \emph{real solutions} of \eqref{eq gZK 4/3}.

The equation in \eqref{eq gZK 4/3} is called $L^2$-critical not only because of $L^2$ being the largest resolution space 
for the IVP  \eqref{eq gZK 4/3}, but also due to the fact that the nonlinearity is the \lq\lq borderline" case that imposes restrictions on the size of the data in 
$L^2$ in order to obtain global solutions in $H^1(\R^3)$. In this case, the energy and the mass are invariant (See \eqref{eq l^2 norm}, 
\eqref{eq energy}). In addition, in the study of the existence and stability of ground state solutions to the equation in \eqref{gen-zk} (see \eqref{ground} below) it
is known that for $k<4/3$ the ground states are stable and are unstable if $k>4/3$ (see \cite{deBouard}).  The case $k=4/3$ is conjectured to be unstable. 

Recently, Farah, Holmes, Roudenko and Yang in \cite{FHRY}  showed that the solutions of the IVP associated to the 2D  $L^2$-critical
equation in \eqref{gen-zk}  blow-up in finite or infinite time.  For ground states theory in the 2D case see \cite{FHR} and \cite{CMPS}.

The facts above motivate our interest in the study of the IVP \eqref{eq gZK 4/3}. Our main goal here is to establish local and global
well-posedness for the IVP \eqref{eq gZK 4/3}. Nevertheless, the techniques above described do not apply directly in the analysis we will implement to prove the local theory.

Our main result is stated as follows.

\begin{theorem}\label{thm-lwp}  Let $u_0 \in H^s, \, s \in (3/4,1)$. Then there are function spaces $X^s_T$ so that 
the IVP \eqref{eq gZK 4/3} has a unique solution 
\[
u \in C([0,T] \colon H^s(\R^3)) \cap X^s_T,
\]
where $T = T(\|u_0\|_s) > 0.$ Moreover, the map $u_0 \mapsto u$ from $H^s(\R^d)$ to $C([0,T] \colon H^s(\R^3))\cap X^s_T$ is locally Lipschitz continuous. 
\end{theorem}

To obtain this result we first employ the methods introduced by Ribaud and Vento \cite{RV1,RV2}, used to deal with the Zakharov-Kuznetsov equation in dimension 3 and the $k-$generalized Zakharov Kuznetsov equation, $k \in \N, k \ge 2,$ in dimension $2.$ 

In order to deal with the fact that our exponent lies strictly between $1$ and $2,$ we need to come up with a method which takes
into consideration both the approach for $k=1$ \cite{RV1} as well as that for $k=2$ \cite{Gru1}. To handle the fact that the nonlinearity is not polynomial, we need to introduce an 
almost-paraproduct decomposition of $u^{7/3}$, together with a nonlinear lemma (see Lemma \ref{lemma nonlinear 1}), in order to deal with that contribution. This nonlinear lemma can be interpreted as an anisotropic version of Proposition 2.3 in \cite{DG}. This allows us to employ the Picard 
iteration method in order to deal with $s \in (3/4,1).$

In the estimates defining the spaces $X^s_T$ to solve the IVP \eqref{eq gZK 4/3} we include a maximal function estimate which is not enough to help us get further regularity but just $s=1^{-}$ at most.  To overcome that problem we prove an \lq\lq almost well-posedness'' result, in the following sense. 
For $s \in [1,2),$ we modify the proof of the nonlinear Lemma \ref{lemma nonlinear 1} in order to obtain a priori estimates 
on the $H^s$ norm of a solution $u(t)$ of  \eqref{eq gZK 4/3}. This is achieved by, instead of using the gradient in the analysis of low frequencies, applying the Laplacian. Unfortunately, this only gives us a priori bounds, and the Duhamel integral operator would only be $1/3-$H\"older
continuous with that method, which makes us use a different method than a direct Picard iteration approach. 
The way out is to consider solutions with smooth enough initial data and lower regularity level, and use the a priori estimate we have in our hands to upgrade this solution by proving it is actually smoother. In the end, we use approximations to the original initial data in $H^s$ and functional analysis considerations to prove that the solution is, in fact, \emph{strong}, and it belongs to the same kind of Banach spaces
measuring regularity we used for the Picard iteration approach.  

\begin{theorem}\label{thm weak at} Let $u_0 \in H^s(\R^3),$ with $s \in [1,2).$ Then there is $T = T(\|u_0\|_s)$ so that the solution $u$ given by Lemma \ref{lemma well under} is the unique
that belongs to to $L^{\infty}([0,T] \colon H^s(\R^3)) \cap \tilde{X}^s_T,$ where we define $\tilde{X}^s_T = \underset{3/4 < r < s}{\bigcap} X^r_T$ and endow it with the norm 
$\| \cdot \|_{\tilde{X}^s_T}^2 =\displaystyle \int_{5/6}^s \|\cdot\|_{X^r_T}^2 \, \mmd r.$ Moreover, if $u_0, v_0 \in H^s$ satisfy that $\|u_0 - v_0\|_{H^s}$ 
is sufficiently small, then so is $\|u-v\|_{\tilde{X}^s_t}.$ 
\end{theorem}

This result still does not give well-posedness in its total strength. Fortunately, as a consequence of Theorem \ref{thm weak at}, we are still able to obtain that the solution 
as stated belongs to the same kind of spaces as in the $s \in (3/4,1)$ case.

\begin{corollary}\label{thm continuity} Let $u_0 \in H^s(\R^3), \, s \in [1,2)$ be as before. Then the solution $u$ to the IVP \eqref{eq gZK 4/3} belongs to $X^s_T.$ In particular, 
$u \in C([0,T] \colon H^s(\R^3)) \cap X^s_T.$ 
\end{corollary}

\begin{remark} In order to put this result in perspective, notice that Theorem \ref{thm-lwp} guarantees the data-to-solution map in the $s \in (3/4,1)$ context is
continuous, while in the remaining range with our current methods we can only guarantee that the solution does not lose regularity and forms a continuous curve 
in $H^s(\R^3).$ 
\end{remark}

\begin{remark} Although the range $s\in(3/4,2)$ looks reasonably far away from the scaling index $s_c = 0$, we mention that this endpoint 
cannot be included in the well-posedness theory, as least not if we demand that the data-to-solution $u_0 \mapsto u$ map is uniformly 
continuous. In Proposition \ref{illposed} it is shown that the
data-to-solution map is not uniformly continuous in $H^s(\R^3), s\le 0$.
\end{remark}

Next we would like to extend our solution globally in $H^1(\R^3)$. To do so we use the smooth solutions found in order to establish 
the previous theorem. We prove that those solutions allow us to justify the conservation of mass and energy satisfied by the
flow of the equation. More precisely, our next main result reads as follows.

\begin{theorem}\label{thm conserve} Let $u_0 \in H^s, \, s \in (3/4,2),$ and denote by $u \in C([0,T] \colon H^s(\R^3)) \cap X^s_T$ the unique solution to the IVP \eqref{eq gZK 4/3} given by Theorem \ref{thm weak at}. 
The following assertions hold: 
\begin{enumerate}  
\item[(i)] For each $t \in [0,T],$ where $T = T(\|u_0\|_s)$ is given by Theorem \ref{thm weak at}, the quantity 
\begin{equation}\label{eq l^2 norm} 
M(u(t)) = \|u(t)\|_2^2
\end{equation}
is preserved by the flow of the solution of \eqref{eq gZK 4/3}.
\item[(ii)] If $s \in (1,2),$ then, for each $t \in [0,T],$ 
the energy 
\begin{equation}\label{eq energy}
E(u(t)) = \int_{\R^3} |\nabla u(t)|^2 \,\mmd x \, \mmd y - \frac{3}{5} \int_{\R^3} u(x,y,t)^{10/3} \, \mmd x \, \mmd y 
\end{equation}
is preserved by the flow of the solution of \eqref{eq gZK 4/3}.
\end{enumerate}
\end{theorem} 

As a consequence of this theorem we obtain our second main result which is the global existence of solutions of the IVP \eqref{eq gZK 4/3}  in $H^1:$

\begin{theorem}\label{thm global} Let $u_0 \in H^1(\R^3)$ and $u \in C([0,T] \colon H^1) \cap X^1_T$ the solution to \eqref{eq gZK 4/3}, as obtained in Theorem \ref{thm weak at} and Corollary 
\ref{thm continuity}. Then, if $\|u_0\|_2$ is sufficiently small, we have that the solution $u \in (L^{\infty} \cap C)(\R \colon H^1(\R^3)).$ 
\end{theorem}

Finally,  we are concerned with the question of almost everywhere (a.e.) convergence of solutions of the IVP \eqref{eq gZK 4/3} to initial data. 
We recall the problem proposed by Carleson in \cite{car} in the context of Schr\"odinger operators: for which $s \in \R$ does it hold that if $u_0\in H^s(\R^n),$  then
\begin{equation}\label{car-1}
\underset{t\to 0}{\lim}\; e^{it\Delta}u_0(x)=u_0(x) \hskip15pt {\rm for \,\, a.e.}\; x?
\end{equation}
He showed that {a.e.}\,convergence holds for $u_0\in H^{1/4}(\R^n)$, $n=1$.
Since then the study of this problem (linear and nonlinear) has gained a lot of attention, and inspired by the recent work of Compaan, Luc\`a and Staffilani \cite{CLS}, we prove that for each initial datum 
$u_0 \in H^s(\R^3), \, s > \frac{3}{4},$ we have \emph{pointwise convergence} of the flow to $u_0.$ 

\begin{theorem}\label{thm pointwise}  Let $s\in (3/4,2)$. Then, for each initial datum $u_0 \in H^s(\R^d),$ the unique solution 
$u \in C([0,T] \colon H^s(\R^3))$ to the IVP \eqref{eq gZK 4/3} given by Theorems \ref{thm-lwp} and \ref{thm weak at}  converges \emph{pointwise} to $u_0;$ i.e.,
\[
\lim_{t \to 0} u(x,t) = u_0(x), \, \text{ for a.e. } x \in \R^3.
\]
\end{theorem}

We use Gr\"unrock's sharp $L^4_{x,y}$ maximal estimate together with an adaptation of our previous techniques to obtain such a result.

\subsection*{Organization} Before moving on to the proofs of our main results, we describe how the paper is organized.

In Section 2, we define the functional spaces we will work, and present linear
estimates useful to define these spaces. In Section 3, we prove Theorem \ref{thm-lwp}. In Section 4, we prove the \lq\lq almost
well-posedness result, Theorem \ref{thm weak at}. 
Finally, in Section 5, we prove the global result (Theorem \ref{thm global}), the conserved quantities (Theorem \ref{thm conserve})
and the Pointwise convergence (Theorem \ref{thm pointwise}).

\section{Preliminaries} 

\subsection{Notation and basic definitions} We use the modified Vinogradov equation $A \lesssim B$ several times to indicate that there is an absolute constant $C>0$ so that $A \le C \cdot  B.$ We also use the original Vinogradov equation
$A \ll B$ to denote that there is a (relatively) large constant $C$ with the property $A \le C \cdot B.$ We also use several times the notation $(-\Delta)^{s/2} f = \langle \nabla \rangle^s f = \mathcal{F}^{-1} ( (1+|\xi|^2+|\eta|^2)^{s/2} \widehat{f}).$ 

Throughout this manuscript, we will work with the inhomogeneous Sobolev spaces $H^s,$ which are the closure of Schwartz functions with 
respect to the norm $\|u\|_{H^s}^2 = \displaystyle\int_{\R^3} |\widehat{u}(\xi)|^2 (1+|\xi|)^{2s} \, \mmd \xi.$ Of course, these function spaces are nested, as can be trivially seen from the definition. \\

We will also use several times the notation 
$$\Delta_j(f)  = \mathcal{F}^{-1}(\varphi_j \cdot \widehat{f}), \, j >0,$$ 
where $\varphi$ is a smooth function supported in the annulus $\{z \in \R^3 \colon |z| \in [1,2]\},$ equal to one on $\{z \in \R^3 \colon |z| \in [5/4,7/4]\},$ $\varphi_j(z) = \varphi(z/2^j),$ and so that 
\[
\Delta_0(f) + \sum_{j > 0} \Delta_j(f) = f,
\]
where we define $\Delta_0(f) := \sum_{j \le 0} \mathcal{F}^{-1}(\varphi_j \widehat{f}).$ For any $j \gg 1,$ the support properties of our $\varphi$ imply that
\[
\tilde{\Delta}_j := \Delta_{j-1} + \Delta_j +\Delta_{j+1} 
\]
satisfies that $\tilde{\Delta}_j \circ \Delta_j = \Delta_j.$ We will use this identity implicitly throughout the text. 

These operators are useful due to the following Littlewood--Paley characterization of Sobolev spaces (see, e.g., \cite[Theorem~1.3.6]{GModern}): 
$f \in H^s(\R^3)$ if and only if
\[
|f|_{s}=\|\Delta_0 (f)\|_2 + \left(\sum_{j > 0} 2^{2sj}\|\Delta_j(f)\|_2^2 \right)^{1/2} < + \infty.
\]
Moreover, the left-hand side above defines an equivalent norm to the usual $\| \cdot \|_{H^s}$ norm. 
With aid of the dyadic decomposition outlined above, we define the main spaces we will use to prove well-posedness. Define, for $u \in \mathcal{S}(\R^4),$ the following norms:
\begin{equation}\label{eq norm sobolev}
\|u\|_{\Xstsobolev} = \|\Delta_0 (u)\|_{L^{\infty}_T L^2_{x,y}} +   \left(\sum_{j > 0} 2^{2sj}\|\Delta_j(u)\|_{L^{\infty}_T L^2_{x,y}}^2 \right)^{1/2},
\end{equation}
\begin{equation}\label{eq norm smooth}
\|u\|_{\Xstsmooth} =  \|\Delta_0 (u)\|_{L^{\infty}_x L^2_{y,T}} +   \left(\sum_{j > 0} 2^{2(s+1)j}\|\Delta_j(u)\|_{L^{\infty}_x L^2_{y,T}}^2 \right)^{1/2},
\end{equation}
\begin{equation}\label{eq norm max}
\|u\|_{\Xstmax} =  \|\Delta_0 (u)\|_{L^2_x L^{\infty}_{y,T}} +   \left(\sum_{j > 0} 2^{2(s-1-\eps)j}\|\Delta_j(u)\|_{L^2_x L^{\infty}_{y,T}}^2 \right)^{1/2},
\end{equation}
\begin{equation}\label{eq norm strichartz} 
\|u\|_{\Xststrichartz} = \|\Delta_0 (u)\|_{L^4_{x,y,T}} +   \left(\sum_{j > 0} 2^{2sj}\|\Delta_j(u)\|_{L^4_{x,y,T}}^2 \right)^{1/2}.
\end{equation}
We therefore define the spaces $X^s_T$ as the closure of $\mathcal{S}(\R^4)$ under the norm
\begin{equation}\label{eq norm main}
\|u\|_{X^s_T} = \|u\|_{\Xstsobolev} + \|u\|_{\Xstsmooth} + \|u\|_{\Xstmax} + \|u\|_{\Xststrichartz}.
\end{equation}
These are the main spaces which will help us prove well-posedness in our setting. We note that sometimes we will use the fact that, for $k > 0,$ we may write 
\[
\|\Delta_k(u)\|_{X^s_T} \sim \|\Delta_k(u)\|_{Y^s_T},
\]
where we let
\[
\|u\|_{Y^s_T} = \|u\|_{L^{\infty}_T H^s_{x,y}} + \| \langle \nabla \rangle^{s+1} u \|_{L^{\infty}_x L^2_{y,T}} + \|\langle \nabla \rangle^{s-1^+} u \|_{L^2_x L^{\infty}_{y,T}} + \| \langle \nabla \rangle^s u\|_{L^4_{x,y,T}}.
\]

\subsection{Linear Estimates} Let 
\begin{align}
\begin{cases}
\partial_t u + \partial_x \Delta u = 0 & \text{ on }\; \R^3 \times \R; \cr
u(x,0) = u_0(x) & \text{ on } \;\R^3, \cr
\end{cases}
\end{align}
be the linear part of \eqref{eq gZK 4/3}. We denote the solution to this problem by $U(t)u_0$
defined by
\begin{equation*}
\widehat{U(t)u_0}(\xi,\eta) = e^{it\xi(\xi^2+|\eta|^2)} \widehat{u_0}(\xi,\eta), \, \xi \in \R, \eta \in \R^{2}.
\end{equation*}
By Plancherel's theorem, this is an unitary group on $L^2.$ Besides that, the following estimates hold: 

\begin{proposition}[Kato Smoothing]\label{thm localsmoothing} For any $u_0 \in L^2(\R^3),$ it holds that 
\[
\|\nabla U(t) u_0\|_{L^{\infty}_x L^2_{y,T}} \lesssim \|u_0\|_{L^2}.
\]
\end{proposition}

\begin{proof} See, for instance, \cite[Proposition~3.1]{RV2}.
\end{proof}

\begin{proposition}[Maximal Estimate]\label{thm max est} For any $T \in (0,1),$ it holds that 
\[
\|U(t)u_0\|_{L^2_x L^{\infty}_{y,T}} \lesssim \|u_0\|_{H^s},
\]
for all $s>1$ and all $u_0 \in \mathcal{S}(\R^3).$ In particular, for the endpoint $s=1,$ it holds that 
\[
\|U(t) \Delta_k u_0\|_{L^2_x L^{\infty}_{y,T}} \lesssim k^2 \|\Delta_k u_0\|_{H^1},
\]
for all $k \ge 1.$ 
\end{proposition}

\begin{proof} See Propositions 3.2 and 3.3 in \cite{RV2}. 
\end{proof}

\begin{proposition}[$L^2-L^4$ Strichartz estimate]\label{thm gruenrock} For all $u_0 \in L^2(\R^3),$ it holds that 
\[
\|U(t)u_0\|_{L^4_{x,y,t}} \lesssim \|u_0\|_{L^2}. 
\]
\end{proposition}

\begin{proof} This is a direct consequence of the observations made in \cite{Gru1}. Alternatively, see also \cite[Lemma~3]{Gru2}. 
\end{proof}

As a consequence of those estimates, we have the following retarded estimates for the group $U(t):$

\begin{proposition}[\cite{RV2}, Propositions 3.5--3.7] Let $f \in \mathcal{S}(\R^4).$ It holds that 
\begin{enumerate}
 \item[(i)]\label{thm retarded unitary}  
 \[
 \left\| \nabla \int_0^t U(t-t') f(t') \, \mmd t'\right\|_{L^{\infty}_T L^{2}_{x,y}} \lesssim \|f\|_{L^1_x L^2_{y,T}};
 \]
 \item[(ii)]\label{thm retarded localsmoothing} 
  \[
 \left\| \nabla^2 \int_0^t U(t-t') f(t') \, \mmd t' \right\|_{L^{\infty}_x L^2_{y,T}} \lesssim \|f\|_{L^1_x L^2_{y,T}};
 \]
 \item[(iii)]\label{thm retarded max est} 
  \[
 \left\| \int_0^t U(t-t') \Delta_k f(t') \, \mmd t' \right \|_{L^2_x L^{\infty}_{y,T}} \lesssim_{\eps} 2^{\eps k} \|\Delta_k f\|_{L^1_x L^2_{y,T}}. 
 \]
\end{enumerate}
\end{proposition}

\begin{proposition}[Retarded $L^2-L^4$ Strichartz]\label{thm retarded gruenrock} Let $f \in \mathcal{S}(\R^4).$ It holds that 
\[
\left\| \nabla \int_0^t U(t-t') f(t') \, \mmd t' \right\|_{L^4_{x,y,t}} \lesssim \|f\|_{L^1_xL^2_{y,t}}.
\]
\end{proposition}

\begin{proof} Proposition \ref{thm gruenrock} together with the dual version of Proposition \ref{thm localsmoothing} implies that 
\[
\left\| \nabla \int_{-\infty}^{\infty} U(t-t') f(t') \, \mmd t' \right\|_{L^4_{x,y,t}} \lesssim \|f\|_{L^1_x L^2_{y,t}}.
\]
Using Lemma B.3, part (i) in \cite{BLC}, which is an anisotropic version of the famous Christ--Kiselev lemma (see \cite{CK}), gives us the result. 
\end{proof}

\section{The Cauchy problem for \eqref{eq gZK 4/3} : The case $s \in (3/4,1)$} 

As mentioned in the introduction, for the easier $s \in (3/4,1)$ range we are able to employ the Picard iteration method. Indeed, we let 
\[
\Gamma(u) = U(t)u_0 + \int_0^t U(t-t') \partial_x(u^{7/3})(t') \, \mmd t'.
\]
We wish to prove that $\Gamma$ preserves the metric space $E_a(T) = \{u \in X^s_T, \|u\|_{X^s_T} \le a\}$ for some $a>0,$ and that, for such $a,$ it also defines a contraction there. In order 
to prove the first fact, we need to prove that 
\begin{equation}\label{eq group bound}
\|U(t)u_0\|_{X^s_T} \lesssim C_s \|u_0\|_{H^s}.
\end{equation}
This is a relatively simple computation using the fact that $U$ is unitary and Propositions \ref{thm localsmoothing},\ref{thm max est} and \ref{thm gruenrock}, and thus we will omit it. Therefore, we are left 
with the integral term $\int_0^t U(t-t') \partial_x(u^{7/3})(t') \, \mmd t'$ to handle. As bounding that term is essentially as difficult as bounding the difference 
\[
\int_0^t U(t-t') \partial_x(u^{7/3}-v^{7/3})(t') \, \mmd t',
\]
which naturally arises when trying to prove that $\Gamma$ is a contraction, we prove only the latter. 

On the other hand, we observe that the retarded Proposition \ref{thm retarded unitary}(i), \ref{thm retarded localsmoothing}(ii), \ref{thm retarded max est}(iii) and 
\ref{thm retarded gruenrock} imply, by the same token as before, that 
\begin{equation}\label{eq nonlinear}
\left\| \int_0^t U(t-t')\partial_x (u^{7/3}-v^{7/3})(t') \, \mmd t' \right\|_{X^s_T} \lesssim \left\| 2^{sk} \|\Delta_k (u^{7/3}-v^{7/3}) \|_{L^1_xL^2_{y,T}} \right\|_{\ell^2(\N)}.
\end{equation}
Until this point, we have followed essentially the same lines as in \cite{RV1,RV2}. The first point we have to change from the approach undertaken in those references is in the decomposition of the nonlinear part. 
In the case of the Zakharov--Kuznetsov equation, Ribaud and Vento's approach involves writing
\[
u^2= \lim_{j \to \infty} (P_j u)^2 = (P_0 u)^2 + \sum_{j \ge 0} \Delta_{j+1} u (P_{j+1} u + P_j u).
\]
Of course, such a clean formula is only available when the exponent is an integer. Our approach will be similar, yet being careful with the fact that we have a fractional nonlinearity. 

\begin{lemma}\label{lemma nonlinear 1} Let $u,v \in \mathcal{S}(\R^4)$ and $k \ge 1.$ It holds that $\| \Delta_k(u^{7/3}-v^{7/3}) \|_{L^1_x L^2_{y,T}}$ is bounded by an absolute constant times
\begin{equation}\label{eq first bound} 
\begin{split}
\left(\| u\|^{4/3}_L +  \| v\|^{4/3}_L\right)& \sum_{m \ge k} \| \Delta_m(u-v) \|_{L^{p_1}_{x,y,T}} \\
+ \left(\| u\|^{1/3}_L + \| v\|^{1/3}_L\right) & \|u-v\|_L \times \sum_{m \ge 0} \min(1,(k-m)_+ 2^{m-k}) \|\Delta_m u\|_{L^{p_1}_{x,y,T}}.
\end{split}
\end{equation} 
Here, we abbreviate $L = L^{4p_1'/3}_x L^{4\tilde{p_1}/3}_{y,T},$ where $q', \tilde{q}$ are defined for $q \ge 2$ so that $\frac{1}{q} + \frac{1}{q'} = 1, \, \frac{1}{q} + \frac{1}{\tilde{q}} = 1/2.$ 
\end{lemma}
\begin{proof} 
We start by writing 
\begin{equation*}
\begin{split}
\Delta_k(u^{7/3}-v^{7/3}) &= \sum_{m \ge 0} [\Delta_k((P_{m+1} u)^{7/3}-(P_{m+1}v)^{7/3}) - \Delta_k ( (P_m u)^{7/3}-(P_m v)^{7/3}))] \\
			  &= \sum_{m \ge 0} \Delta_k (J_m(u,v)),
\end{split}
\end{equation*}
where we define $J_m(u,v)$ to be
\[
 P_{m+1}(u-v) \int_{0}^1 ((P_{m+1}u) + \theta (P_{m+1}(v-u)))^{4/3} 
 \]
\[
- P_m(u-v) \int_0^1 ((P_{m}u) + \theta (P_{m}(v-u)))^{4/3} \, \mmd \theta.
\]
We modify this expression once more: it can be rewritten as 
\begin{align*}
 & \Delta_{m+1}(u-v) \int_0^1 ((P_{m+1}u) + \theta (P_{m+1}(v-u)))^{4/3} \mmd \theta \cr
 & -  P_m(u-v) \int_0^1 \left[ ((P_{m}u) + \theta (P_{m}(v-u)))^{4/3} - ((P_{m+1}u) + \theta (P_{m+1}(v-u)))^{4/3}\right] \mmd \theta .\cr
\end{align*}
Finally, the last expression can be written in the form
\begin{align*}
 &=:  \Delta_{m+1}(u-v) \tilde{J}_m(u,v) + c P_m(u-v) \int_0^1 [\Delta_{m+1} u + \theta \Delta_{m+1}(u-v)] \cr  
 & \left( \int_0^1 \left[ ((P_{m}u) + \theta (P_{m}(v-u))) + s(\Delta_{m+1} u + \theta \Delta_{m+1}(u-v))\right]^{1/3} \mmd s \right) \mmd \theta. \cr
\end{align*}
By virtue of Minkowski's inequality, Young's convolution inequality and H\"older's inequality, we get that, for $m \ge k,$ 
\begin{equation*}
\begin{split}
 & \| \Delta_k(J_m(u,v)) \|_{L^1_x L^2_{y,T}} \\
					  & \lesssim  \|\Delta_{m+1} (u-v) \|_{L^{p_1}_{x,y,T}} \left(\|u\|_{L^{4p_1'/3}_x L^{4\tilde{p_1}/3}_{y,T}}^{4/3} + \|v\|_{L^{4p_1'/3}_x L^{4\tilde{p_1}/3}_{y,T}}^{4/3}\right) \\
					  & +   \| \Delta_m u \|_{L^{p_1}_{x,y,T}} \left(\|u\|_{L^{4p_1'/3}_x L^{4\tilde{p_1}/3}_{y,T}}^{1/3} + \|v\|_{L^{4p_1'/3}_x L^{4\tilde{p_1}/3}_{y,T}}^{1/3}\right)\|u-v\|_{L^{4p_1'/3}_xL^{4\tilde{p_1}/3}_{y,T}}.
\end{split}					  
\end{equation*}
Now, for $ m < k,$ we gain an exponential factor by introducing a gradient to take advantage of frequency localization around $2^k.$ 
Young's inequality several times plus H\"older's inequality then imply
\begin{align*}
& \| \Delta_k (J_m (u,v)) \|_{L^1_x L^2_{y,T}} \cr 
&\lesssim 2^{-k} \| \nabla [ (P_{m+1} u)^{7/3}-(P_{m+1}v)^{7/3} ]  - \nabla [(P_m u)^{7/3}-(P_mu)^{7/3}] \|_{L^1_xL^2_{y,T}} \cr 
					 \lesssim 2^{-k} \sum_{i=0,1} & \| |P_{m+i}(u-v)|(|\nabla P_{m+i}u|+|\nabla P_{m+i}v|)|(|P_{m+i}u|^{1/3}+|P_{m+i} v|^{1/3}\|_{L^1_xL^2_{y,T}}. \cr
\end{align*}
This is clearly bounded by an absolute constant times
\[
2^{-k} \left( \sum_{m' \le m+1}  2^{m'} \| \Delta_{m'} u \|_{L^{p_1}_{x,y,T}} \right) \| u-v\|_{L^{4p_1'/3}_xL^{4\tilde{p_1}/3}_{y,T}} \left(\|u\|_{L^{4p_1'/3}_x L^{4\tilde{p_1}/3}_{y,T}}^{1/3} + \|v\|_{L^{4p_1'/3}_x L^{4\tilde{p_1}/3}_{y,T}}^{1/3}\right). 
\]
Together with the bound for high frequencies this implies the lemma. 
\end{proof}

We note that this lemma is heavily inspired by Proposition 2.3 in \cite{DG}, which is a clean version of such a result in the recent literature. 

\subsection{Analysis of high frequencies} We now look at the frequencies $\ge k$ in the bound given by Lemma \ref{lemma nonlinear 1}. More precisely, we wish to estimate
\begin{equation}\label{eq high}
\|w\|_{L^{4p_1'/3}_xL^{4\tilde{p_1}/3}_{y,T}} \hskip5pt\text{ and } \hskip5pt\left\| 2^{sk} \sum_{m > k} \| \Delta_m w \|_{L^{p_1}_{T,x,y}} \right\|_{\ell^2(\N)}
\end{equation}
in terms of the norm $\|w\|_{X^s_T}.$ By interpolating the first and fourth terms in the sum defining $\|\Delta_k w \|_{Y^s_T}$, it holds that 
\begin{equation}\label{eq interpol1}
2^{sk} \| \Delta_k w\|_{L^q_T L^p_{x,y}} \lesssim \|\Delta_k w\|_{Y^s_T},
\end{equation}
whenever we have $\frac{1}{q} = \frac{\theta}{4}, \, \frac{1}{p} = \frac{1}{2} - \frac{\theta}{4}.$ Therefore, 
$$2^{sj} \|\Delta_j w\|_{L^{p_1}_{T,x,y}} \lesssim T^{\delta} 2^{sj}\| \Delta_j w \|_{L^{q_1}_TL^{p_1}_{x,y}} \lesssim T^{\delta} \|\Delta_j w \|_{Y^s_T},$$
for some $\delta >0.$ In order to take care of  $\|w\|_{L^{4(p_1)'/3}_x L^{4(\tilde{p_1})/3}_{y,T}}$, we first bound it by 
\[
\sum_{j \ge 0} \| \Delta_j w\|_{{L^{4(p_1)'/3}_x L^{4(\tilde{p_1})/3}_{y,T}}}.
\]
Now, interpolating the second and third terms defining $Y^s_T,$ we get 
\begin{equation}\label{eq interpol2} 
2^{(s+2\theta-1-\eps) j} \| (\Delta_j w) \|_{L^{q_2}_x L^{4\tilde{p_1}/3}_{y,T}} \lesssim \| \Delta_j w\|_{Y^s_T},
\end{equation}
where $\frac{3}{2\tilde{p_1}} = \theta,$ and thus $q_2 = \frac{8p_1}{p_1 + 6}.$ Bernstein's inequality then implies that
\[
\| \Delta_j w\|_{{L^{4(p_1)'/3}_x L^{4(\tilde{p_1})/3}_{y,T}}} \lesssim 2^{\left(\frac{12-5p_1}{8p_1}\right)j} \| \Delta_j w \|_{L^{q_2}_x L^{4\tilde{p_1}/3}_{y,T}}.
\]
 This implies that 
\[
\|w\|_{L^{4(p_1)'/3}_x L^{4(\tilde{p_1})/3}_{y,T}} \lesssim \sum_{j \ge 0} 2^{\left[\left(\frac{12-5p_1}{8p_1}\right)-(s+2\theta-1-\eps) \right] j}\|w\|_{Y^s_T}.
\]
The sum above converges for $s > 1 - 2\theta - \frac{12-5p_1}{8p_1} = \frac{12+p_1}{8p_1}.$
Notice that the argument above for using Bernstein's inequality implies, in particular, that the same strategy works as long as $p_1 < \frac{12}{5}.$ Therefore, setting $p_1 = \frac{12}{5} - \gamma$ above with $\gamma>0$ sufficiently 
small and running the argument, 
we will see that the second equation in \eqref{eq high} is bounded pointwise by 
\[
T^{\delta} \|(1_{j \le 0} 2^{-sj})*\|(\Delta_j w)\|_{Y^s_T}\|_{\ell^2(\N)} \times \|w\|_{X^s_T}^{4/3} \lesssim T^{\delta} \|w\|_{X^s_T}^{7/3},
\]
by the discrete version of Young's convolution inequality, as long as $s>\frac{3}{4}.$ 

\subsection{Analysis of low frequencies}
We need now to bound
\begin{equation}\label{eq low} 
\left\| 2^{sk} \sum_{k>m} (k-m)_+ 2^{-(k-m)} \|\Delta_m u\|_{L^{p_1}_{T,x,y}}\right\|_{\ell^2(\N)}.
\end{equation}
As we saw in the analysis for the high frequencies, 
\[
\| \Delta_m u\|_{L^{p_1}_{T,x,y}} \lesssim 2^{-sm} T^{\delta} \|\Delta_m u\|_{Y^s_T}.
\]
Thus we get that
\begin{equation}\label{eq young2}
\begin{split}
2^{sk} \sum_{k>m} (k-m) 2^{-(k-m)} \|\Delta_m u\|_{L^{p_1}_{T,x,y}} 
&\lesssim T^{\delta}\sum_{k>m} (k-m) 2^{(s-1)(k-m)} \|\Delta_m u \|_{Y^s_T} \\
&= T^{\delta} \left[(1_{j \ge 0} j 2^{(s-1)j}) * \| \Delta_j u\|_{Y^s_T}\right] (k). 
\end{split}
\end{equation}
Again by the discrete version of Young's inequality, we get that the $\ell^2(\N)$ norm of the expression on the left hand side of \eqref{eq young2} is bounded by 
\[
T^{\delta} \|u\|_{X^s_T} \| (1_{j \ge 0} j 2^{(s-1)j})\|_{\ell^1(\N)} \lesssim T^{\delta} \|u\|_{X^s_T},
\]
as long as $s < 1.$ 
Therefore, we get from \eqref{eq first bound} that 
\[
\left\| \int_0^t U(t-t')\partial_x (u^{7/3})(t') \, \mmd t' \right\|_{X^s_T} \lesssim T^{\delta} \|u\|_{X^s_T}^{7/3},
\]
given that $s \in (3/4,1).$ 

Therefore, together with the estimate \eqref{eq first bound}, we obtain 
\begin{equation}\label{eq final bound}
\left\| 2^{sk} \| \Delta_k(u^{7/3}-v^{7/3})\|_{L^1_xL^2_{y,T}}\right\|_{\ell^2(\N)} \lesssim T^{\delta} \|u-v\|_{X^s_T} \left(\|u\|_{X^s_T}^{4/3} + \|v\|_{X^s_T}^{4/3}\right).
\end{equation}

\subsection{Conclusion} The Duhamel formulation of \eqref{eq gZK 4/3} implies that the operator 
\[
\Gamma(u) = U(t)u_0 + \int_0^t U(t-t')\partial_x(u^{7/3})(t') \, \mmd t' 
\]
is a contraction on $E_a(T),$ where we pick $a = 2C_s\|u_0\|_{H^s}$ the upper bound in \eqref{eq group bound} and $T \sim \|u_0\|_s^{-\beta},$ for some $\beta >0.$ Therefore, 
the fixed point theorem yields us that there is a solution to such a problem in $C([0,T] \colon H^s(\R^3))\cap X^s_T$ with $s \in (3/4,1),$ and it is unique. Moreover, 
standard considerations imply that the data-to-solution map is Lipschitz in that case, which concludes the proof of well-posedness. 

\section{The Cauchy problem for \eqref{eq gZK 4/3}: The case $s \in [1,2).$} 

\subsection{An a priori bound} 

The gain of the form $2^{m-k}$ for low frequencies given by Lemma \ref{lemma nonlinear 1} is enough to allow us to prove local well-posedness for $s \in (3/4,1),$ but it is clear from the 
analysis of low frequencies in \S 3.2 above that it breaks down for $s=1.$ In order to fix that, we will need the following version of Lemma \ref{lemma nonlinear 1}.

\begin{lemma}\label{lemma nonlinear 2} Let $u \in \mathcal{S}(\R^4)$ and $k \ge 1.$ It holds that $\| \Delta_k(u^{7/3}) \|_{L^1_x L^2_{y,T}}$ is bounded by an absolute constant times
\begin{equation}\label{eq mod bound} 
\begin{split}
\| u\|^{4/3}_{X^s_T} \times & \left ( \sum_{m \ge k} \| \Delta_m(u) \|_{L^{p_1}_{x,y,T}} +  \sum_{m \le k} (k-m)_+ 2^{\min(2,\frac{1}{4} + s^-)(m-k)} \|\Delta_m u\|_{L^{p_1}_{x,y,T}}\right).
\end{split} 
\end{equation}
Here, we abbreviate $L = L^{4p_1'/3}_x L^{4\tilde{p_1}/3}_{y,T},$ where $q', \tilde{q}$ are defined for $q \ge 2$ so that $\frac{1}{q} + \frac{1}{q'} = 1, \, \frac{1}{q} + \frac{1}{\tilde{q}} = \frac12,$ 
and $s \in (1-\gamma,+\infty)$ for some sufficiently small $\gamma > 0.$ 
\end{lemma}

\begin{proof} The outline of the proof is exactly the same as in that of Lemma \ref{lemma nonlinear 1}. If we write $\Delta_k(u^{7/3}) = \sum_{m \ge 0} \Delta_k(J_m u)$ as before, the same analysis 
as before yields that, for $m \ge k,$ 
\[
\| \Delta_k (J_m u) \|_{L^1_xL^2_{y,T}} \lesssim \|u\|_L^{4/3} \sum_{m \ge k} \|\Delta_m u\|_{L^{p_1}_{x,y,T}}. 
\]
On the other hand, for the frequencies $m < k,$ we write 
\begin{align*}
\| \Delta_k (J_m u) \|_{L^1_xL^2_{y,T}} &\lesssim 2^{-2k} \sum_{i=0,1} \| \nabla^2 (P_{m+i} u)^{7/3} \|_{L^1_xL^2_{y,T}} \cr 
					\lesssim 2^{-2k} \sum_{i=0,1} & \left(\| (\nabla^2 P_{m+i}u) (P_{m+i} u)^{4/3}\|_{L^1_xL^2_{y,T}}  + \| |\nabla P_{m+i} u|^2 (P_{m+i} u)^{1/3}\|_{L^1_x L^2_{y,T}}\right). \cr 
					=: I_1(k,m) +  & I_2(k,m).  \cr  
\end{align*}
In order to analyze the first term, we use H\"older's inequality in conjunction with Young's convolution and Bernstein's inequality again. This implies that 
\[
I_1(k,m) \lesssim \left(\sum_{m' \le m} 2^{2(m'-k)} \|\Delta_{m'} u \|_{L^{p_1}_{x,y,T}}\right) \|u\|_L^{4/3}.
\]
Using Fubini shows that this term can be absorbed into the right hand side of \eqref{eq mod bound}. Now, for the second term, we use H\"older's inequality for $L^{p_1'}_x L^{\tilde{p_1}}_{y,T}$ with $(\nabla P_{m+i} u) (P_{m+i}u)^{1/3}$ 
and $L^{p_1}_{x,y,T}$ with $\nabla P_{m+i} u.$ Further uses of triangle inequality and Bernstein's inequality yield
\begin{equation}\label{eq second term} 
I_2(k,m) \lesssim 2^{-2k} \left( \sum_{m' \le m} 2^{m'} \|\Delta_{m'} u\|_{L^{p_1}_{x,y,T}}\right) \| (\nabla P_{m} u) (P_{m}u)^{1/3}\|_{L^{p_1'}_x L^{\tilde{p_1}}_{y,T}}.
\end{equation}
Another use of H\"older's inequality shows that the right-hand side above is bounded by 
\[
C 2^{-2k} \left( \sum_{m' \le m} 2^{m'} \|\Delta_{m'} u\|_{L^{p_1}_{x,y,T}}\right) \|u\|_{L}^{1/3} \|\nabla P_m u\|_{L^{4(p_1)'/3}_x L^{4\tilde{p_1}/3}_{y,T}}.
\]
We focus on the last factor above. An analysis identical to the one employed in \S 3.1 implies that 
\[
\|\nabla P_m u\|_{L^{4(p_1)'/3}_x L^{4\tilde{p_1}/3}_{y,T}} \lesssim \sum_{m' \le m} 2^{\left[\left(\frac{12-5p_1}{8p_1}\right)-(s+2\theta-1-\eps) + 1\right] m'} \| \Delta_{m'} u \|_{Y^s_T}.
\]
with $\theta = \frac{3}{2 \tilde{p_1}}.$ Choosing $p_1$ sufficiently close to $12/5,$ the right hand side above is bounded by 
\[
\sum_{m' \le m} 2^{\left(\frac{3}{4} +(1-s) + \eps\right)m'} \| \Delta_{m'} u \|_{Y^s_T} \lesssim 2^{\left(\frac{3}{4} +(1-s) + \eps\right)k} \|u\|_{X^s_T}.
\]
for any $s \in (1-\gamma,+\infty),$  $\gamma>0$ sufficiently small. Using \eqref{eq second term}, the upper bound above and Fubini, 
\[
\sum_{m \in \N, m \le k} I_2(k,m) \lesssim \left(\sum_{m' \le k} (k-m')  2^{\left(\frac{1}{4} + s - \eps\right)(m'-k)} \|\Delta_{m'} u\|_{L^{p_1}_{x,y,T}}\right) \|u\|_{X^s_T}^{4/3}.
\]
As this estimate can clearly be absorbed into the right-hand side of \eqref{eq mod bound}, the proof of the lemma is done. 
\end{proof}

We are now in a position to prove the following a priori bound:

\begin{lemma}\label{lemma a priori} Suppose $u \in X^s_T$ satisfies $\|u\|_{X^s_T} \le 2C_s \|u_0\|_{H^s}$ 
and $s \in (3/4,2)$. Then there is $T' = T'(\|u_0\|_s) < T$ so that 
\[
\|\Gamma^j u\|_{X^s_{T''}} \le 2C_s \|u_0\|_s,
\]
for all $T'' \in [0,T'],$ where $\{\Gamma^k u\}_{k \ge 0}$ denotes the orbit of the function $u$ under the action of $\Gamma.$ 
\end{lemma}

\begin{proof} We notice that the result holds trivially true for $s \in (3/4,1)$ by the considerations in the previous section, and therefore we assume $s \in [1,2).$ 

We write the Duhamel formula defining $\Gamma:$
\begin{equation}\label{eq duhamel} 
\Gamma u(t) = U(t)u_0 + \int_0^t U(t-t') \partial_x (u^{7/3}) (t') \, \mmd t'. 
\end{equation}
Applying the $X^s_{T''}$ norm on both sides, we see that 
\[
\|\Gamma u\|_{X^s_{T''}} \lesssim C_s \|u_0\|_s + \left\| 2^{sk} \| \Delta_k (u^{7/3}) \|_{L^1_xL^2_{y,T''}} \right\|_{\ell^2(\N)}. 
\]
Due to Lemma \ref{lemma nonlinear 2} and the bounds from \S 3.2, we can bound this last expression by 
\[
C_s\|u_0\|_s + (T'')^{\delta} \|u\|_{X^s_{T''}}^{4/3}\left\| 2^{sk} \left( \sum_{m \in \N} \min\left(1,(k-m)_{+}2^{\min(2,\frac{1}{4}+s^{-})(m-k)}\right) \| \Delta_m u\|_{L^q_{T''}L^{p_1}_{x,y}}\right)\right\|_{\ell^2(\N)}. 
\]
Again using that $\|\Delta_m u\|_{L^q_T L^{p_1}{x,y}} \lesssim 2^{-sm}\|\Delta_m u\|_{X^s_{T''}}$ we get that the $\ell^2(\N)$ norm is, in turn, bounded by 
\[
(T'')^{\delta} \|u\|_{X^s_{T''}}^{4/3} \left( \| \|\Delta_j u\|_{X^s_{T''}} * (1_{j \le 0} 2^{sj})\|_{\ell^2(\N)} + \| \|\Delta_j u\|_{X^s_{T''}} * (1_{j \ge 0} j 2^{(s-\min(2,1/4+s^-))j})\|_{\ell^2(\N)}\right).
\]
Now, Young's convolution inequality allows us to bound the expression above by 
$$(T'')^{\delta} \|u\|_{X^s_{T''}}^{7/3}$$
if $s < 2.$ Indeed, if $s < 7/4,$ then $\min(2,1/4+s^-) = 1 + s^-,$ and thus $s - 1/4 - s^- =-1/4^+.$ If $2 > s \ge 7/4,$ we will just 
have $s-2 < 0$ in the exponent of the term being convolved. 

Therefore, if $T'=T'(\|u_0\|_s)$ is sufficiently small, as $\|u\|_{X^s_{T}}  \le 2C_s \|u_0\|_{H^s}$,  it holds that $\|\Gamma u\|_{X^s_{T'}} \le 2C_s \|u_0\|_{H^s}$ as 
well. The assertion of the lemma follows then by iterating this result. 
\end{proof}

\subsection{Weak well-posedness for $s \in [1,2)$}

The following lemma states an `almost well-posedness' for the remaining range. It will be useful in order to prove the persistence of regularity and uniqueness for initial data in $H^s, \, s \in [1,2).$ 

\begin{lemma}\label{lemma well under} Let $u_0 \in H^s(\R^3),$ with $s \in [1,2).$ Then there is $T=T(\|u_0\|_s)$ and a unique solution $u$ to \eqref{eq gZK 4/3} 
so that $u \in C([0,T] \colon H^r(\R^3)) \cap X^r_T,$ for each $r \in (3/4,s).$ 
\end{lemma}

\begin{proof} For $s \in [1,2),$ let $\tilde{s} \in (3/4,1)$ be a fixed parameter. Let $T = T(\|u_0\|_{\tilde{s}},\|u_0\|_s)$ be sufficiently small, and 
$w \in X^s_T$ satisfy that $\|w\|_{X^s_T} \le 2\min(C_{\tilde{s}},C_s) \|u_0\|_{\tilde{s}}.$ The fact that the norms$\|\cdot\|_{X^s_T}$ are increasing 
with $s$ and that we have proved Lipschitz well-posedness for $\tilde{s} \in (3/4,1)$ implies that 
\[
\| \Gamma^j w - \Gamma^k w\|_{X^{\tilde{s}}_T} \to 0 \text{ as } \min(k,j) \to \infty. 
\]
By Lemma \ref{lemma a priori}, we know that 
\[
\| \Gamma^j w - \Gamma^k w\|_{X^s_T} \le 4C_s \|u_0\|_s.
\]
By interpolating between the norms of $X^s_T$ and $X^{\tilde{s}}_T,$ we obtain that 
\[
\| \Gamma^j w - \Gamma^k w\|_{X^r_T} \to 0 \text{ as } \min(j,k) \to \infty
\]
for all $r \in [\tilde{s}, s).$ This implies that the sequence $\{\Gamma^j w\}_{j \ge 0}$ converges in $X^r_T,$ and thus it converges to a fixed point of $\Gamma,$ 
as the proof of Lemma \ref{lemma a priori} also shows that $\Gamma$ is $1/3-$H\"older continuous on $X^r_T,$  $r \in [1,2).$

Let this fixed point be $w \in X^s_T.$ The properties of $\Gamma$ imply that $w \in C([0,T] \colon H^r) \cap X^r_T.$ By the definition of $\Gamma,$ this $w$ satisfies the Duhamel formulation of \eqref{eq gZK 4/3}, 
and thus it is a solution to \eqref{eq gZK 4/3} with initial data in $H^r(\R^3).$ 

If $w_1,w_2 \in C([0,T] \colon H^r) \cap X^r_T$ are solutions to \eqref{eq gZK 4/3} for some $T>0,$ then, by the fact that the spaces $H^r$ and $X^r_T$ are nested
in $r \in \R,$ they are also strong solutions to \eqref{eq gZK 4/3} in $C([0,T] \colon H^{\tilde{s}}) \cap X^{\tilde{s}}_T,$ for some $\tilde{s} \in (3/4,1).$ As we know that uniqueness holds for that case, 
it also holds for $r \in [1,2).$ This finishes the proof. 
\end{proof} 

As a direct consequence, we are in position to prove the almost well-posedness result for $s \in [1,2)$  stated in Theorem \ref{thm weak at}.


\begin{proof}[Proof of Theorem \ref{thm weak at}] Start by considering the set of mollified initial data $u_{0,\eps} = \varphi_{\eps} * u_0$ and $u_{\eps}$ the solution given by Lemma \ref{lemma well under}. By that result, 
it holds that $u_{\eps}$ belongs to $C([0,T_{\eps,\delta}] \colon H^{s+\delta}) \cap X^{s+\delta}_{T_{\eps,\delta}},$ where $T_{\eps,\delta} = T \left(\eps^{-\delta}\|u_0\|_s\right).$ Indeed, 
it holds that 
\begin{equation}\label{eq up inter}
\|u_{\eps}\|_{X^{s+\delta}_{T'}} \le 2C_s \|u_{0,\eps}\|_{s+\delta}.
\end{equation}
whenever $T' < T(\|u_{0,\eps}\|_{s+\delta}).$ On the other hand, as $\|u_{0,\eps}\|_{s+\delta} \lesssim  \eps^{-\delta} \|u_0\|_s$ and the $T$ from Lemma \ref{lemma well under} is nonincreasing on 
the norm of $\|u_0\|_s,$ the conclusion follows. Furthermore, notice that for $\delta \sim \eps,$ we get that $T_{\eps,\delta} \sim T(\|u_0\|_s)$ is \emph{independent} of $\eps,$ like the upper bound 
in \eqref{eq up inter}.

On the other hand, for $\tilde{s} \in (3/4,1),$ the considerations in \S 3 imply that 
\begin{equation}\label{eq down inter} 
\|u_{\eps_1} - u_{\eps_2} \|_{X^{\tilde{s}}_{T''}} \le 2C_{\tilde{s}} \|u_{0,\eps_1} - u_{0,\eps_2}\|_{\tilde{s}},
\end{equation}
whenever $T'' < \tilde{T}(\|u_0\|_{\tilde{s}}).$ By \eqref{eq up inter} and the fact that the $X^s_T-$norms are increasing in $s \in \R,$ it holds that 
$\{u_{\eps}\}_{\eps > 0}$ is a bounded sequence in $X^s_{T'}, T'=T'(\|u_0\|_s).$ Interpolating \eqref{eq up inter} together with \eqref{eq down inter} then implies that $\{u_{\eps}\}_{\eps > 0}$ is a Cauchy sequence in $X^r_{T''},$ 
for all $ r \in (3/4,s)$ and all $T'' < \min(T_0(\|u_0\|_s),\tilde{T}(\|u_0\|_{\tilde{s}})).$ 

By uniqueness of the limit, the solution $u$ to the IVP \eqref{eq gZK 4/3} belongs to all $X^r_{T''}$ for such $r$ and $T''.$ Notice, moreover, that the upper bound on the 
$X^r_{T''}-$norm of the solution is uniformly bounded for such $r$ and $T''.$ This readily implies that $u \in \tilde{X}^s_T.$ 

Also, notice that $\{u_{\eps}\}_{\eps >0}$ is a bounded sequence in $C([0,T_0] \colon H^s),$ and by the Banach--Alaoglu theorem, we can suppose (passing to a subsequence if necessary) 
that $u_{\eps} \rightharpoonup u$  in $L^2([0,T] \colon H^s(\R^3)).$  By properties of weakly convergent sequences and the uniform boundedness of $u_{\eps}$ in $C([0,T] \colon H^s),$ 
we can easily conclude that $u \in L^{\infty}([0,T_0] \colon H^s(\R^3)).$ 

Finally, the continuity of the solution with respect to initial data in the $\tilde{X}^s_T$ norm can be proved as follows: let $u_0, v_0 \in H^s(\R^3), \, s \in [1,2),$ and let 
$u,v \in L^{\infty}([0,\delta] \colon H^s) \cap \tilde{X}^s_{\delta},$ for some common time $\delta >0.$ Then 
\begin{equation}\label{eq continuity} 
\|u-v\|_{X^r_{\delta}} \le \|u-u_{\eps}\|_{X^r_{\delta}} + \|u_{\eps} - v_{\eps}\|_{X^r_{\delta}} + \|v-v_{\eps}\|_{X^r_{\delta}}
\end{equation}
holds trivially by triangle's inequality. Fix then $\eta > 0$ and let $r \in [3/4+\eta,s-\eta].$ By the fact that $(u_{\eps},v_{\eps}) \to (u,v)$ in 
$X^r_{\delta} \times X^r_{\delta},$ and that we can quantify the rate of such a convergence by the interpolation of \eqref{eq up inter} and \eqref{eq down inter}, 
it holds that we can make the first and third terms in \eqref{eq continuity} uniformly small (with respect to $r \in [3/4+\eta,s-\eta]$) by choosing $\eps > 0$ sufficiently small. 

On the other hand, for fixed $\eps > 0,$ the same interpolation argument shows that the middle term $\|u_{\eps} - v_{\eps}\|_{X^r_{\delta}}$ can be bounded in such an interval of values of $r$ 
by some positive power of $\|u_{0,\eps} - v_{0,\eps}\|_{H^r} \lesssim C\|u_0 - v_0\|_{H^r}.$ If $v_0$ and $u_0$ are close, then so is the middle term of \eqref{eq continuity}. Therefore, 
\[
\int_{5/6}^s \|u-v\|_{X^r_{\delta}}^2 \, \mmd r \lesssim (s-\eta - 5/6) \times \delta + \eta \times \sup_{r \in [s-\eta,s]} \left(\|u\|_{X^r_{\delta}}^2 + \|v\|_{X^r_{\delta}}^2\right).
\]
By the uniform boundedness of the $X^r_{\delta}$ norms asserted above, the right-hand side in this last equation can be made arbitrarily small. 
\end{proof}



\begin{proof}[Proof of Corollary \ref{thm continuity}] We first prove a useful property of weak convergence: \\

\noindent\textit{Claim: If $u_{\eps} \rightharpoonup u$ in $L^2([0,T] \colon H^s(\R^3)),$ then $\Delta_k(u_{\eps}) \rightharpoonup \Delta_k(u)$ in the same space.} 
Indeed, let $v \in L^2([0,T] \colon H^{-s}(\R^3))$ be fixed. Then 
\[
\int_0^T \langle v(t), \Delta_k u_{\eps}(t) \rangle_{H^{-s},H^s} \, \mmd t = \int_0^T \langle \Delta_k v(t) , u_{\eps} (t) \rangle_{H^{-s},H^s} \, \mmd t.
\]
As $\Delta_kv \in  L^2([0,T] \colon H^{-s}(\R^3)),$ the righ-hand side of the last equation still converges to 
\[
\int_0^T \langle \Delta_k v(t), u(t) \rangle_{H^{-s},H^s} \, \mmd t = \int_0^T \langle v(t) , \Delta_k u(t) \rangle_{H^{-s},H^s} \, \mmd t
\]
as $\eps \to 0.$ This concludes the proof of this claim. \\ 

Therefore, fix $K \in \N.$ By passing to subsequences a finite number of times, we can suppose that, besides the weak convergence property given by the claim above, we have that 
$\Delta_k(u_{\eps}) \to \Delta_k(u)$ for almost every $(x,y,t) \in \R^3 \times [0,T], \,\, k =1,2,\dots,K.$ In order to bound $\|u\|_{X^s_T},$ we need to bound terms of the form 
\begin{equation}\label{eq problematic} 
\|\Delta_k(u)\|_{L^{\infty}_x L^2_{y,T}} , \,\|\Delta_k(u)\|_{L^2_x L^{\infty}_{y,T}} , \, \|\Delta_k(u)\|_{L^{\infty}_T L^2_{x,y}}, \text{ and} \, \|\Delta_k(u)\|_{L^4_{x,y,T}}.
\end{equation}
The easiest to handle is the last term in \eqref{eq problematic}. Indeed, 
\[
\int_{\R^3 \times [0,T]} |\Delta_k(u)(x,y,t)|^4 \, \mmd x \, \mmd y \, \mmd t \le \liminf_{\eps \to 0} \int_{\R^3 \times [0,T]} |\Delta_k(u_{\eps})(x,y,t)|^4 \, \mmd x \, \mmd y \, \mmd t
\]
follows directly from Fatou's lemma and the asserted pointwise convergence from before. The strategy to bound the other terms is similar, but this time we need a substitute for the $L^{\infty}$
norms appearing. Recall, for that purpose, that the normalized norms $\|\cdot\|_{\L{}^N (I,\mmd m)} = \frac{1}{m(I)^{1/N}} \| \cdot \|_{L^N(I,\mmd m)}$ are bounded by the $L^{\infty}-$norm on $I.$ 
Moreover, if $f \in L^{\infty},$ then $\|f\|_{\L{}^N} \nearrow \|f\|_{\infty}$ as $N \to \infty.$ Inspired by that, we can prove, for instance, that for fixed $R> 0, N \gg 1,$ 
\begin{equation}\label{eq approximate infinity} 
\left( \int_{\R} \left(\dashint_{B_R^2 \times [0,T]} |\Delta_k(u)(x,y,t)|^N \, \mmd y \, \mmd t \right)^{2/N} \, \mmd x \right)^{1/2} 
\end{equation}
is bounded -- by the same Fatou argument as before -- by 
\[
\liminf_{\eps \to 0} \left( \int_{\R} \left(\dashint_{B_R^2 \times [0,T]} |\Delta_k(u_{\eps})(x,y,t)|^N \, \mmd y \, \mmd t \right)^{2/N} \, \mmd x \right)^{1/2} \le \liminf_{\eps \to 0} \|\Delta_k(u_{\eps})\|_{L^2_x L^{\infty}_{y,T}}.
\]
On the other hand, taking first $N \to \infty$ in \eqref{eq approximate infinity} and using the fact that $L^N-$averages converge monotonically to the supremum and monotone convergence, then taking $R \to \infty$ and using monotone convergence again, 
we see that \eqref{eq approximate infinity} converges to $\|\Delta_k(u)\|_{L^2_x L^{\infty}_{y,T}}.$ As the other terms in \eqref{eq problematic} can be bounded in similar (in fact, simpler) 
ways, we have that, for each fixed $K \in \N,$ the bound
\[
\|P_K u\|_{X^s_T} \le \limsup_{\eps \to 0} \|u_{\eps}\|_{X^s_T} \le 2C_s\|u_0\|_s
\]
holds, for $T = T(\|u_0\|_s)$ sufficiently small, given in Theorem \ref{thm weak at}. Taking $K \to \infty$ above then implies that $u \in X^s_T, \, \|u\|_{X^s_T} \le 2C_s \|u_0\|_s.$ Writing the Duhamel formulation of 
\eqref{eq gZK 4/3} that $u$ satisfies as 
\[
u(t) - U(t)u_0 = \int_0^t U(t-t') \partial_x (u^{7/3})(t') \, \mmd t'
\]
and applying the $X^s_T-$norm on both sides and the estimates from Lemma \ref{lemma a priori}, we get continuity of $u(t)$ at $t=0.$ By the group property, continuity at all other values of $t \in [0,T]$ follow in 
the same way.
\end{proof}

Although the range $s \in (3/4,2)$ looks reasonably far away from the scaling index $s_c = 0,$ we mention that this endpoint cannot be included in the well-posedness theory, as least not if we demand that the data-to-solution 
$u_0 \mapsto u$ map is uniformly continuous. This is the content of the following proposition. 

\begin{proposition}\label{illposed} Let $s \le 0.$ Then the IVP \eqref{eq gZK 4/3} is ill-posed for initial data in $H^s(\R^3),$ in the sense that the data-to-solution map 
is not uniformly continuous. 
\end{proposition}

\begin{proof} We only sketch the proof, as the main outline is the same as in \cite[Theorem~1.2]{LPa}. 

Indeed, fix $\phi \in H^1(\R^3)$ a positive, radial solution to 
\begin{equation}\label{ground}
\Delta \phi - \phi + \phi^{7/3} = 0.
\end{equation}
The existence of such a function is a by-product of the theory developed by Weinstein \cite{W}. Now, we let $\psi_c(r) = c^{3/4} \phi(c^{1/2} r),$  where we abuse notation and denote 
a radial function by its value at radius $r\ge0.$ The function 
\[
u_c(z,t) = \psi_c((x-ct)^2 + |y|^2), \,\, z=(x,y) \in \R^3,
\]
can be seen through an easy computation to satisfy the IVP \eqref{eq gZK 4/3} with initial datum $u_{0,c}(z) = c^{3/4} \phi(\sqrt{c}x).$
Moreover, it is easy to verify that the Fourier transform of the initial data $\widehat{u_{0,c}}(\xi,\eta) = c^{-3/4} \widehat{\phi}\left(\frac{\xi}{\sqrt{c}},\frac{\eta}{\sqrt{c}}\right).$ A simple calculation then shows that
\[
\lim_{n \to \infty} \langle u_{0,c_n},u_{0,c_{n+1}} \rangle_{L^2} = \|\phi\|_{L^2}^2,
\]
where $c_n = n, \, \forall n \ge 1.$ As we also have that $\|u_{0,c_n}\|_{L^2} = \|\phi\|_{L^2},$ we conclude that 
\[
\|u_{0,c_n} - u_{0,c_{n+1}}\|_{L^2} \to 0 \text{ as } n \to \infty.
\]
On the other hand, a similiar computation with the Fourier transform for fixed time $t>0$ yields that 
\[
\langle u_{c_n}(t), u_{c_{n+1}}(t) \rangle_{L^2} = \left(\frac{n}{n+1}\right)^{3/4} \int_{\R^3} e^{2\pi i \xi \sqrt{n}} \widehat{\phi}(\xi,\eta) \overline{\widehat{\phi}}\left(\frac{\sqrt{n}\xi}{\sqrt{n+1}},\frac{\sqrt{n}\eta}{\sqrt{n+1}}\right) \, \mmd \xi \mmd \eta.
\]
An application of the Riemann--Lebesgue lemma implies that the expression above goes to $0$ as $n \to \infty.$ By the way we defined our solutions, we have conservation of $L^2$ norm, and thus $\|u_{c_n}(t)\|_2 = \|\phi\|_2.$ This implies 
that 
\[
\lim_{n \to \infty} \|u_{c_n}(t) - u_{c_{n+1}}(t)\|_{L^2} = \sqrt{2} \|\phi\|_2.
\]
Thus the solution map is not uniformly continuous.
\end{proof}

\section{Conservation laws and global existence in $H^s$} 




\subsection{The auxiliary problem} In order to prove that the conservation laws \eqref{eq l^2 norm} and \eqref{eq energy}, we 
are going to be interested in the following smoother version of \eqref{eq gZK 4/3}:

\begin{equation}\label{eq gZK 4/3 smooth} 
\begin{cases}
\partial_t u^{\eps} + \partial_x \Delta u^{\eps} + \partial_x(F^{7/3}_{\eps}(u^{\eps})) = 0 & \text{ on } \R^3 \times \R; \cr
u^{\eps}(x,0) = u_0(x) & \text{ on } \R^3,
\end{cases}
\end{equation}
where we define $F_{\eps}^\alpha (u) = \eta_{\eps} * (u * \eta_{\eps})^{\alpha} ,$ for some smooth, even, positive function $\eta:\R^d \to \R$. We notice that this approach towards verifying the 
validity of the conservation laws is far from new; see, for instance, the pioneering work of Ginibre and Velo \cite{GV}, the quantification due to Weinstein \cite{W} and the book by Tao \cite{Tao} 
for further details on this method. \\

The first result about this equation is that solutions to \eqref{eq gZK 4/3 smooth} are smooth for smooth initial data, as the following proposition states: 

\begin{proposition}\label{prop smooth} For each $u_0 \in H^N(\R^3), \, N \gg 1,$ there exists $T=T(\eps,N,\|u\|_{H^N})$ such that the IVP \eqref{eq gZK 4/3 smooth} is locally well-posed in 
$$C([0,T] \colon H^N(\R^3)) \cap L^{\infty}([0,T] \colon W^{1,\infty}(\R^3)).$$ 
\end{proposition}

\begin{proof} The proof involves a standard parabolic regularization method. Let 
\begin{equation}\label{eq parabolic} 
\begin{cases}
\partial_t w  + \eta (\Delta^2 w) + \partial_x \Delta w + \partial_x(F^{7/3}_{\eps}(w)) = 0 & \text{ on } \R^3 \times \R; \cr
w(x,0) = u_0(x) & \text{ on } \R^3,
\end{cases}
\end{equation}
where we take $\eta >0.$ The fact that solutions of \eqref{eq parabolic} are smooth can be proved easily with the contraction principle, using the fact that we may write the Duhamel formulation 
of the equation as
\[
w(t) = e^{-\eta t \Delta^2} u_0 + \int_0^t e^{-\eta (t-t')\Delta^2} (\partial_x \Delta w + \partial_x (F_{\eps}^{7/3}u))(t') \, \mmd t'
\]
and simple $H^s$ estimates for the group $\{e^{-t\eta \Delta^2}\}_{t \in \R}$ given on the Fourier side. For more details see, for instance, \cite[Section~3.1]{Erdoganbook} or even \cite{Faminskii}. Of course, this method gives us a maximal existence time depending on $\eta>0.$ In order to remove this restriction, 
we apply $\Delta^k$ to \eqref{eq parabolic} 
and integrate the resulting equation against $\Delta^k w.$ We obtain that 
\[
\partial_t \|\Delta^k w(t)\|_2^2 = - 2 \int (\Delta^k w) \partial_x (\Delta^{k+1} w) - 2 \eta \int (\Delta^k w) (\Delta^2 w) - 2 \int \partial_x \Delta^k (F^{7/3}_{\eps} w) (\Delta^k w).
\]
The first term on the right hand side above is zero, and the second is non-positive. Therefore, we are left with bounding 
\begin{equation}\label{eq third term}
\left| \int \partial_x \Delta^k (F_{\eps}^{7/3} w) (\Delta^k w) \right|.
\end{equation}
By the way we defined the smooth nonlinearity, we obtain that $\partial_x \Delta^k (F_{\eps}^{7/3} w) = (\Delta^k (\eta_{\eps})) * \partial_x(w * \eta_{\eps})^{7/3}.$ This implies that \eqref{eq third term} 
is bounded by 
$$C_k \eps^{-2k-1} \|\Delta^k w\|_2 \|(w * \eta_{\eps})^{4/3} (\partial_x w)* \eta_{\eps} \|_2 \lesssim C_k \eps^{-2k-1} \| \Delta^k w\|_2 \|w\|_{H^1}^{7/3}.$$
Putting this together with what we had before, summing in $2k \le N$ and using first the Gagliardo--Nirenberg interpolation inequality, and then the Cauchy--Schwarz inequality plus properties of the convolution, we obtain
\[
\partial_t \|w(t)\|_{H^N}^2 \le C_{N}(\eps)\|w(t)\|_{H^N} \|w\|_{H^1}^{7/3} \le  C_N(\eps) \|w(t)\|_{H^N}^2 \|w\|_{H^1}^{4/3}.
\]
The standard Gronwall trick implies that 
\begin{equation}\label{eq gronwall}
\|w(t)\|_{H^N} \le \|u_0\|_{H^N} \cdot \exp\left(C_N(\eps) \int_0^t \|w(t')\|_{H^1}^{4/3} \, \mmd t'\right).
\end{equation}
Notice that this last equation does not depend on $\eta,$ and thus it implies that there is $T_0 = T_0(\|u_0\|_{H^N}, N,\eps)$ but not depending on $\eta$
so that $\|w(t)\|_{H^N} \le 2\|u_0\|_{H^N}$ for $t \in [0,T_0].$ This observation implies that the maximum time $T_N$ of solution of \eqref{eq parabolic} so that 
$u \in C^1([0,T_N] \colon H^N)$ for $u_0 \in H^N$ does not depend on $\eta>0.$ 

These estimates, together with the very structure of the equation \eqref{eq parabolic}, enables us to use the Bona-Smith argument \cite{BS}. This implies that, if $w^{\eta,\eps}$ denotes
the solution to \eqref{eq parabolic}, then taking $\eta \to 0,$ $w^{\eta,\eps} \to u^{\eps}$ in $H^N.$ Moreover, because of \eqref{eq gronwall}, the fact that solutions to \eqref{eq parabolic} are smooth
and the fact the maximal time does not depend on $\eta,$ it holds that solutions to \eqref{eq gZK 4/3 smooth} are also smooth for smooth enough initial data, and the problem is also well-posed on high order Sobolev 
spaces. We omit the details, as they are essentially completely contained in the references previously mentioned. 
\end{proof}

We have proved that, for each fixed $\eps >0,$ the solution of \eqref{eq gZK 4/3 smooth} preserves regularity. This will be useful for proving that the solutions to \eqref{eq gZK 4/3} satisfy the conserved quantities. 

On the other hand, in order to do so, we need to prove that solutions of \eqref{eq gZK 4/3 smooth} are well-defined for the range of Sobolev spaces in which we are working. 
This is the content of the next result. 

\begin{theorem}\label{thm smooth sol} Let $u_0 \in H^s(\R^3).$ The following assertions hold: 
\begin{enumerate}
 \item[(i)] If $s \in (3/4,1),$ then there is $T=T(\|u_0\|_s)$ \emph{independent} of $\eps$ so that \eqref{eq gZK 4/3 smooth} is well-posed in 
 $C([0,T] \colon H^s) \cap X^s_T$ with initial data $u_0.$ 
\item[(ii)] If $s \in [1,2),$ then there is $T = T(\|u_0\|_s, \eps)$ so that \eqref{eq gZK 4/3 smooth} is well-posed in 
 $C([0,T] \colon H^s) \cap X^s_T$ with initial data $u_0.$ 
\end{enumerate}
In both cases, the data-to-solution map is Lipschitz continuous.
\end{theorem}

\begin{proof}[Proof of part (i).] Define the map
\[
\Gamma_{\eps}(w) = U(t)u_0 + \int_0^t U(t-t')\partial_x (F_{\eps}^{7/3}w)(t') \, \mmd t'.
\]
Exactly in the same way as we did for the IVP \eqref{eq gZK 4/3}, it suffices to control 
\[
\left\|\int_0^t U(t-t')\partial_x (F_{\eps}^{7/3}w_1-F_{\eps}^{7/3}w_2)(t') \, \mmd t'\right\|_{X^s_T} 
\]
\[
\lesssim \left\|2^{sk} \|\Delta_k(\eta_{\eps}*((w_1 * \eta_{\eps})^{7/3}-(w_2 * \eta_{\eps}) ^{7/3}))\|_{L^1_xL^2_{y,T}} \right\|_{\ell^2(\N)}.
\]
As convolution with $\eta_{\eps}$ commutes with $\Delta_k,$ an application of Young's convolution inequality shows that the right hand side above is bounded by 
\[
\left\|2^{sk} \|\Delta_k((w_1 * \eta_{\eps})^{7/3} - (w_2 * \eta_{\eps})^{7/3})\|_{L^1_xL^2_{y,T}} \right\|_{\ell^2(\N)}.
\]
By the analysis undertaken in \S 3, the quantity above is controlled by $C T^{\delta}\|w_1 - w_2\|_{X^s_T}.$ Therefore, the same argument works for proving well-posedness in $C([0,T] \colon H^s) \cap X^s_T,$ 
$s \in (3/4,1),$ of the IVP \eqref{eq gZK 4/3 smooth}. Notice that, as none of the constants involved depended on $\eps > 0,$ the existence time does \emph{not} depend on $\eps > 0$ for such range of $s.$ \\ 

\noindent \textit{Proof of part (ii).} We follow the same path to existence as  before, but now we need an analogue of Lemma \ref{lemma nonlinear 2} for differences. We follow thus the 
outline of proof of Lemma \ref{lemma nonlinear 1}. The $m \ge k$ part can be kept without any modification. On the other hand, in order to estimate the contributions of 
terms $\| \Delta_k (J^{\eps}_m (w_1,w_2)) \|_{L^1_x L^2_{y,T}}$ with $m < k,$ we notice that 
\[
\nabla^2 (\eta_{\eps} * [ (P_m (w_1 * \eta_{\eps}))^{7/3} - (P_m (w_2 * \eta_{\eps}))^{7/3}])
\]
\[
= (\nabla \eta_{\eps}) * \left[\nabla  [ (P_m (w_1 * \eta_{\eps}))^{7/3} - (P_m (w_2 * \eta_{\eps}))^{7/3}] \right].
\]
Therefore, employing the same techniques as in both of those cases, the low-frequency contribution $\Delta_k(J^{\eps}_m(w_1,w_2))$ can be bounded by 
\[
\frac{C}{\eps} (\|w_1\|_L^{1/3} + \|w_2\|_L^{1/3}) \cdot \|w_1 - w_2\|_L \times \left( \sum_{m < k} (k-m)_+ 2^{2(m-k)} \|\Delta_m u\|_{L^{p_1}_{x,y,T}}\right).
\]
The rest of the proof can be carried out in the exact same way as before. 
\end{proof}

Notice that redoing the proof of Lemmata \ref{lemma nonlinear 1} 
in \ref{lemma nonlinear 2} yields only a $\|u-v\|_{X^s_T}^{1/3}$ factor, and thus the operator $\Gamma$ from \S 3 is only $1/3-$H\"older continuous. For that reason we need to work with the smooth versions \eqref{eq gZK 4/3 smooth} 
and then hope to be able to pass to the limit. The next corollary is a major step in that direction. 

\begin{corollary}\label{thm uniform time} Let $u_0 \in H^s(\R^3)$ for $s \in [1,2).$ Denote the solution to \eqref{eq gZK 4/3 smooth} obtained in Theorem \ref{thm smooth sol} by $u^{\eps}.$ Then $u^{\eps}$ extends to a time
$T = T(\|u_0\|_s)$ \emph{independent} of $\eps > 0.$ 
\end{corollary}

\begin{proof} Using Theorem \ref{thm smooth sol}, part (ii), we know that there is $T_{\eps}$ so that $u^{\eps} \in C([0,T_{\eps}] \colon H^s) \cap X^s_{T_{\eps}}.$ 
On the other hand, it is not hard to verify that Theorem \ref{thm weak at} and its proof generalize almost verbatim to solutions of \eqref{eq gZK 4/3 smooth} withouth that the time $T$ for which 
the solution belongs to $L^{\infty}([0,T] \colon H^s(\R^3))$ depends on $\eps>0.$ Therefore, by iterating Theorem \ref{thm smooth sol} part (ii), we conclude the desired assertion. 
\end{proof}

\subsection{Proof of Theorem \ref{thm conserve}} Finally, we prove that the $L^2$ norm \eqref{eq l^2 norm} is preserved by the flow of \eqref{eq gZK 4/3} for all $s \in (3/4,2),$ and that 
the energy \eqref{eq energy} is conserved for $s \in (1,2).$ 

More specifically, we already know by the results in \S 3 and Lemma \ref{lemma well under} and Theorem \ref{thm weak at} that, whenever $u_0 \in H^s(\R^3)$, then there exists a solution 
$u \in C([0,T] \colon H^s(\R^3)) \cap X^s_T, \, T=T(\|u_0\|_s),$ and this solution is unique with such properties. We will prove that 

\begin{lemma}\label{lemma sub convergence} Let $u,u^{\eps}$ be defined as before. Then, for each $s \in (3/4,1),$ there is $T=T(\|u_0\|_s)$ so that $u^{\eps} \to u$ in $X^s_T$ as $\eps \to 0$.
\end{lemma}

\begin{proof} 
Using the Duhamel formulation of \eqref{eq gZK 4/3} and \eqref{eq gZK 4/3 smooth}, we can conclude that 
\[
\|u^{\eps} - u\|_{X^s_T} \lesssim \left\| \int_0^t U(t-t')\partial_x(F_{\eps}^{7/3}(u^{\eps}) - u^{7/3})(t') \, \mmd t' \right\|_{X^s_T}.
\]
Now, the estimate \eqref{eq nonlinear}, the second term on the right-hand side above is bounded by 
\begin{equation}\label{eq approx1}
\left\| 2^{sk} \|\Delta_k(\eta_{\eps}*((u^{\eps} * \eta_{\eps})^{7/3} - (u * \eta_{\eps})^{7/3}))\|_{L^1_xL^2_{y,T}} \right\|_{\ell^2(\N)} 
\end{equation}
\begin{equation}\label{eq approx2}
+ \left\| 2^{sk} \|\Delta_k(\eta_{\eps}*((u * \eta_{\eps})^{7/3} - u^{7/3}))\|_{L^1_xL^2_{y,T}} \right\|_{\ell^2(\N)} 
\end{equation}
\begin{equation}\label{eq approx3}
+ \left\| 2^{sk} \|\Delta_k(\eta_{\eps}*(u^{7/3}) - u^{7/3})\|_{L^1_xL^2_{y,T}} \right\|_{\ell^2(\N)}.
\end{equation}

The first term \eqref{eq approx1} is clearly bounded by $C T^{\delta} \|u^{\eps} - u\|_{X^s_T}$ by the methods from before, while the \eqref{eq approx2} is 
bounded by $C T^{\delta} \|(u * \eta_{\eps}) - u\|_{X^s_T},$ which converges to zero by approximate identity properties, and 
\eqref{eq approx3} converges to zero by the fact that the summand within the $\ell^2(\N)-$norm is 
pointwise bounded by $2 \times 2^{sk} \|\Delta_k(u^{7/3})\|_{X^s_T},$ and by the properties of approximate identities converges pointwise to zero, which enables us to use the dominated convergence theorem. 

This implies that $u^{\eps} \to u$ in $X^s_T \cap C([0,T] \colon H^s(\R^3)), \, s \in (3/4,1),$ as desired. 
\end{proof} 

Next, we prove the conservation law for solutions of \eqref{eq gZK 4/3 smooth}:

\begin{lemma}\label{lemma conserve approx} Let $u_0 \in H^s(\R^3), \, s \in (3/4,2),$ and $u^{\eps}$ be the solution to \eqref{eq gZK 4/3 smooth} given by the previous methods. The following assertions hold: 
\begin{enumerate}
 \item The quantity 
 \begin{equation}\label{eq conserve approx} 
 M(u^{\eps}(t)) = \|u^{\eps}(t)\|_2^2
 \end{equation}
 is constant for $t \in [0,T].$ 
 \item If, moreover, $s \in [1,2),$ we have that the quantity 
 \begin{equation}\label{eq conserve approx2}
 E_{\eps}(u^{\eps}(t)) = \| \nabla u^{\eps}(t)\|_2^2 - \frac{3}{5}\|(u^{\eps}(t) * \eta_{\eps})\|_{10/3}^{10/3}
 \end{equation}
 is constant for $t \in [0,T].$ 
\end{enumerate}

\end{lemma}

\begin{proof} We first note that the proposition holds whenever we are dealing with \emph{smooth} solutions to \eqref{eq gZK 4/3 smooth}. Indeed, for (1), multiplying that equation by $u^{\eps}$ and integrating on $\R^3$ yields, together
with a couple of integrations by parts, that the derivative of \eqref{eq conserve approx} is zero, which yields the claim. For (2), we integrate \eqref{eq gZK 4/3 smooth} against $\Delta u^{\eps} + F_{\eps}^{7/3}(u^{\eps})$ and integrations
by parts in the same way as previously prescribed yields the result. 

In order to pass to low regularity, we consider the sets 
\[
S_1 = \left\{ t \in [0,T] \colon \|u^{\eps}(t)\|_2^2  = \|u_0\|_2^2\right\}
\]
and
\[
S_2 = \left\{ t \in [0,T] \colon \|\nabla u^{\eps}(t)\|_2^2 - \frac{3}{5}\|u^{\eps}(t) * \eta_{\eps}\|_{10/3}^{10/3} = 
 \|\nabla u_0\|_2^2 - \frac{3}{5}\|u_0\|_{10/3}^{10/3} \right\}.
\]

As we know that $u^{\eps} \in C([0,T] \colon H^s)$ whenever $s \in (3/4,2),$ then the expression defining \eqref{eq conserve approx} can be shown to be continuous with $t,$ and thus $S_1 \subset [0,T]$ is closed. 
All we have to do it to show that it is also open. But this is a standard connectivity argument: if $t_1 \in S_1 \cap (0,T),$ then considering the IVP \eqref{eq gZK 4/3 smooth} 
with initial value $u^{\eps}(t_1) * \varphi_{\delta}$ shows that the formal manipulations that we need in order to conclude that \eqref{eq conserve approx} hold rigorously for this new IVP 
on a neighbourhood of $t_1,$ independently of $\delta,$ by \eqref{eq gronwall}. Taking a limit as $\delta \to 0$ implies the desired result. 

Similarly for $S_2,$ the expression in \eqref{eq conserve approx2} is continuous in time for $t \in [0,T],$ by Theorem \ref{thm smooth sol}, Corollary \ref{thm uniform time} and the Gagliardo--Nirenberg interpolation inequality. As the set 
$S_2 \subset [0,T]$ is closed, we just need to conclude it is open once more. But this follows for $s \in [1,2)$ by well-posedness of \eqref{eq gZK 4/3 smooth} in that range, together with and approximation argument as before. We 
skip the details.
\end{proof}

\begin{proof}[Proof of Theorem \ref{thm conserve}] \textit{Part (i):} Let first $s \in (3/4,1).$ Then, it holds that, whenever $u_0 \in H^s,$ 
\[
\|u^{\eps}(t)\|_2 = \|u_0\|_2^2, \hskip10pt  \forall \, t \in [0,T].
\]
By Theorem \ref{thm smooth sol}, the existence time $T$ above depends only on the $H^s-$norm of the initial datum. Now, by Lemma \ref{lemma sub convergence}, we have that, for such $T,$ 
\[
\|u^{\eps} - u\|_{X^s_T} \to 0 \, \Rightarrow \|u_0\|_2^2 = \|u^{\eps}(t)\|_2^2 \to \|u(t)\|_2^2 \; \text{ as }\; \eps \to 0.
\]
This finishes the proof in this case. For $s \in [1,2),$ we simply use that $H^s$ are nested, and thus we can reduce to the previous case. \\

\noindent\textit{Part (ii):} Assume first, by mollification, that $u_0 \in H^3(\R^3).$ We know, from Theorem \ref{thm weak at} applied to solutions of \eqref{eq gZK 4/3} and \eqref{eq gZK 4/3 smooth},  Corollary \ref{thm continuity} and Lemma \ref{lemma sub convergence}, 
that
\[
\|u^{\eps} - u\|_{X^s_T} \lesssim 4C_s\|u_0\|_s,
\]
whenever $s \in [1,2),$ and 
\[
\|u^{\eps} - u\|_{X^{\tilde{s}}_T} \to 0, \, \, \tilde{s} \in (3/4,1), \,
\]
as $\eps \to 0,$ and $T = T(\|u_0\|_s)$ is sufficiently small. Interpolating between these two estimates implies that $u^{\eps} \to u$ in $X^s_T \cap C([0,T] \colon H^s)$ whenever 
$s \in [1,2).$ Therefore, whenever $s \in [1,2),$ we see that 
\[
E_{\eps}(u_0) = E_{\eps}(u^{\eps}(t)) \to E(u(t))\, \text{ as } \eps \to 0 \text{ for } \, t \in [0,T]
\]
follows from the aforementioned convergence and the Gagliardo--Nirenberg interpolation inequality. The left-hand side above, however, converges by approximate identity properties to $E(u_0).$ 

Now, for general $u_0 \in H^s(\R^3), \, s \in (1,2),$ we solve \eqref{eq gZK 4/3} for initial data $u_{0,\eps} = u_0 * \eta_{\eps}.$ Using the notation from Theorem \ref{thm weak at}, it holds that
\begin{equation}\label{eq almost there}
E(u_{0,\eps}) = E(u_{\eps}(t)), \, \forall \, t \in [0,T].
\end{equation}
On the other hand, from Theorem \ref{thm weak at} itself and properties of approximate identities, the left-hand side of \eqref{eq almost there} converges to $E(u_0),$ while the right-hand side 
converges to $E(u(t)),$ for each $t \in [0,T]$ as long as $s \in (1,2).$ This concludes the proof of the result. 
\end{proof}

\subsection{Proof of Theorem \ref{thm global}} In order to prove Theorem \ref{thm global}, we need to work once more with the approximations $u_{0,\eps}$ to our initial data. Theorem \ref{thm conserve} shows that
the solutions $u_{\eps}$ satisfy the conservation of energy \eqref{eq energy} for $t \in [0,T], T = T(\|u_{0,\eps}\|_1) \gtrsim T(\|u_0\|_1).$ Therefore, for $t \in [0,T],$ 
\[
\| \nabla u_{\eps}(t)\|_2^2 \le \|u_0\|_1^2 + \frac{3}{5}\|u_{\eps}(t)\|_{10/3}^{10/3} \le \|u_0\|_1^2 + \frac{3}{5} \,C_{opt}^{10/3} \|\nabla u_{\eps}(t)\|_2^2 \|u_0\|_2^{4/3},
\]
where $C_{opt}$ denotes the best constant in the Gagliardo--Nirenberg interpolation inequality 
\[
\|w\|_{10/3} \le C_{opt} \|\nabla w\|_2^{3/5} \|w\|_2^{2/5}.
\]
Therefore, if $\|u_0\|_2 < \left(\frac{5}{3}\right)^{3/4} \frac{1}{C_{opt}^{5/2}},$ using conservation of \eqref{eq l^2 norm} for $u_{\eps},$ we get that 
\begin{equation}\label{eq uniform all}
\|u_{\eps}(t)\|_{H^1} \le C_{\|u_0\|_2,\|u_0\|_{H^1}},
\end{equation}
for all times $t$ in which $u_{\eps}$ is well-defined. 
This proves already that, by reiterating Theorem \ref{thm weak at} and Corollary \ref{thm continuity}, 
$u_{\eps} \in (C\cap L^{\infty})(\R \colon H^1(\R^3)).$ In particular, by Lemma \ref{lemma a priori}, there is another universal constant $\tilde{C}_{\|u_0\|_2,\|u_0\|_{H^1}}$ so that, for each $t_0 \in \R,$ 
\begin{equation}\label{eq uniform all 2} 
\| u_{\eps}( \cdot + t_0)\|_{X^1_T} \le \tilde{C}_{\|u_0\|_2,\|u_0\|_{H^1}},
\end{equation}
for each $T$ sufficiently small depending only on $\|u_0\|_2, \|u_0\|_{H^1}.$ Applying \eqref{eq uniform all} and 
the standard functional analysis arguments mentioned before, we get that the solution $u$ extends to all the real line as a function in $L^{\infty}(\R \colon H^s).$ In addition to it, 
by Corollary \ref{thm continuity} applied to \eqref{eq gZK 4/3} with initial data $u(t_0)$ iteratively, we get that 
\[
\|u(\cdot + t_0 )\|_{X^1_T} \le \tilde{C}_{\|u_0\|_2,\|u_0\|_{H^1}} . 
\]
This implies, in particular, that $u \in C([t_0,t_0+T] \colon H^1(\R^3)),$ and as $t_0 \in \R$ was arbitrary, it holds that $u \in (L^{\infty} \cap C)(\R \colon H^1(\R^3)),$ as desired. 

We remark that, by the results in \cite{W} (see also \cite[Chapter~6]{LP}), we have that $\frac{1}{C_{opt}} = \frac{\|\phi\|_2^{2/5}}{(5/3)^{3/10}},$ where 
$\phi \in H^1(\R^3)$ is some positive, radial solution to 
\begin{equation}\label{eq elliptic} 
\Delta \phi - \phi + \phi^{7/3} = 0.
\end{equation}
Therefore, the above arguments work whenever 
\[
\|u_0\|_2 < \|\phi\|_2. 
\]
This gives us an explicit quantification of how small the initial data has to be in order to have global existence results. 

\subsection{Consequences of the analysis of \eqref{eq gZK 4/3 smooth}: Pointwise convergence to the initial data} At last, we prove that for each $u_0 \in H^s, \, s> \frac{3}{4},$ 
the flow of \eqref{eq gZK 4/3} converges \emph{pointwise} to $u_0$ as $t \to 0.$ In order to do so, we employ classical ideas from the theory of maximal functions, which have
been recently employed in the context of nonlinear Schr\"odinger equations by Compaan, Luc\`a and Staffilani \cite{CLS}. See also \cite{ET1, ET2} for previous work by Erdo\u{g}an and Tzirakis 
on convergence results for the KdV equation and the nonlinear Schr\"odinger in one dimension. 

We begin with a standard lemma translating maximal estimates into pointwise consequences. Here we note that we will work with \eqref{eq gZK 4/3 smooth}, but we pick carefully the function $\eta$ and the 
initial data. In fact, we can choose it so that $\eta_{\eps} * f (z) = P_{1/\eps} f(z),$ with $P_N = \sum_{j \le N} \Delta_j$ defined as before, and we let $u_{0,\eps} = \eta_{\eps} * u_0$ be the 
initial data we work with. We will abuse notation and still denote by $u^{\eps}$ the solution to this new version of \eqref{eq gZK 4/3 smooth}. 

Notice that this specification does not change any of the previously discussed properties of the solutions $u^{\eps}$ to \eqref{eq gZK 4/3 smooth}.

 \begin{lemma}\label{lemma maximal} Suppose that, for $u_0 \in H^s(\R^3),$ we have that 
\[
\|u^{\eps} - u\|_{L^4_{x,y} L^{\infty}_T} \to 0 \text{ as } \eps \to 0.
\]
Then it holds that $u(z,t) \to u_0(z)$ as $t \to 0$ for almost every $z \in \R^3.$ 
\end{lemma}

\begin{proof} The proof is analogous to that of \cite[Lemma~3.8]{LinaresRamos2}, so we omit it. 
\end{proof}

\begin{proof}[Proof of Theorem \ref{thm pointwise}] We focus on proving the asserted convergence from Lemma \ref{lemma maximal}. Indeed, 
an analogous strategy to that of Proposition \ref{thm retarded max est} (iii) implies 
\[
\left\| \int_0^t U(t-t') \partial_x \Delta_k f(t') \, \mmd t' \right\|_{L^4_{x,y} L^{\infty}_T} \lesssim 2^{(3/4)^+ k} \| \Delta_k f\|_{L^1_x L^2_{y,T}}. 
\]
By the Duhamel formulations of \eqref{eq gZK 4/3} and \eqref{eq gZK 4/3 smooth}, we get that $\|u - u^{\eps}\|_{L^4_{x,y} L^{\infty}_T}$ is bounded by 
\begin{align}\label{eq bound point}
C_s \|(I-P_{1/\eps})u\|_{H^s}  + \left\| 2^{sj} \| \Delta_j (\eta_{\eps} * ( u^{\eps} * \eta_{\eps})^{7/3} - u^{7/3})\|_{L^1_x L^2_{y,T}}\right\|_{\ell^2_j},
\end{align}
where $s > \frac{3}{4}.$ By using the same decomposition as in Lemma \ref{lemma sub convergence}, the second term in \eqref{eq bound point} is bounded by the sum 
\[
\eqref{eq approx1} + \eqref{eq approx2} + \eqref{eq approx3} \to 0 \text{ as } \eps \to 0,
\]
where we argue exactly as in Lemma \ref{lemma sub convergence}, now keeping in mind that the correction term $\|(I-P_{1/\eps}) u_0\|_{H^s} \to 0$ as $\eps \to 0$ if 
$u_0 \in H^s(\R^3).$ This concludes the proof, and by Lemma \ref{lemma maximal} we deduce the desired pointwise convergence of the flow.
\end{proof}

\section*{Acknowledgements}
F. L. was partially supported by CNPq and FAPERJ, Brazil. J.P.G.R. acknowledges financial support from CNPq, Brazil.

\end{document}